\documentclass[aap]{imsart}

\RequirePackage{amsthm,amsmath,amsfonts,amssymb}
\RequirePackage[numbers,sort&compress]{natbib}
\RequirePackage[colorlinks,citecolor=blue,urlcolor=blue]{hyperref}
\RequirePackage{graphicx}
\RequirePackage{multirow}

\startlocaldefs
\theoremstyle{plain}

\newtheorem{theorem}{Theorem}[section]
\newtheorem{lemma}[theorem]{Lemma}
\theoremstyle{remark}
\newtheorem{definition}[theorem]{Definition}
\newtheorem*{example}{Example}

\newcommand{\ceil}[1]{\lceil #1 \rceil}

\endlocaldefs

\begin{document}

\begin{frontmatter}
\title{Emergence of Multivariate Extremes in Multilayer Inhomogeneous Random Graphs}
\runtitle{Multivariate Extremes in Multilayer Random Graphs}

\begin{aug}
\author[A]{\fnms{Daniel}~\snm{Cirkovic}\ead[label=e1]{cirkovd@stat.tamu.edu}},
\author[B]{\fnms{Tiandong}~\snm{Wang}\ead[label=e2]{td\textunderscore wang@fudan.edu.cn}}
\and
\author[A]{\fnms{Daren} B.H.~\snm{Cline}\ead[label=e3]{dcline@stat.tamu.edu}}
\address[A]{Department of Statistics,
 Texas A\&M University \printead[presep={,\ }]{e1,e3}}

\address[B]{Shanghai Center for Mathematical Sciences,
Fudan University \printead[presep={,\ }]{e2}}
\end{aug}

\begin{abstract}
In this paper, we propose a multilayer inhomogeneous random graph model (MIRG), whose layers may consist of both single-edge and multi-edge graphs. In the single layer case, it has been shown that the regular variation of the weight distribution underlying the inhomogeneous random graph implies the regular variation of the typical degree distribution. We extend this correspondence to the multilayer case by showing that the multivariate regular variation of the weight distribution implies the multivariate regular variation of the asymptotic degree distribution. Furthermore, in certain circumstances, the extremal dependence structure present in the weight distribution will be adopted by the asymptotic degree distribution. By considering the asymptotic degree distribution, a wider class of Chung-Lu and Norros-Reittu graphs may be incorporated into the MIRG layers. Additionally, we prove consistency of the Hill estimator when applied to degrees of the MIRG that have a tail index greater than 1. Simulation results indicate that, in practice, hidden regular variation may be consistently detected from an observed MIRG. 
\end{abstract}



\end{frontmatter}

\section{Introduction}

The scale-free phenomenon, or the notion that many real-world networks tend to exhibit power-law degree distributions, is a fundamental hypothesis in network science that is supported by significant amount of empirical evidence \cite{mislove2007measurement, voitalov2019scale}. In fact, the prominence of scale-free networks is so well-established that estimates of the power-law tail index regularly appear as summary statistics for networks listed in network repositories such as KONECT \cite{kunegis2013konect}. Such findings have led researchers to develop network models that imitate scale-free behavior. Prominent examples include the preferential attachment model and the inhomogeneous random graph \cite{barabasi1999emergence, newman2003structure, bollobas2007phase, van2013critical}. 

Often research in network science makes inquires regarding large degree-degree relationships when edges are of differing types. For example, in a X (formerly Twitter) network, does a user having a large number of followers indicate that they may receive have a large number of likes or retweets? Similarly, in forum-like websites such as Reddit, is a large number of links in one sub-forum necessarily associated with a large number of links in another sub-forum? How does this relationship vary by sub-forum type? Such questions are naturally analyzed in a multilayer network setting.

We study large degree-degree relationships in a multilayer inhomogeneous random graph (MIRG) model through the lens of multivariate regular variation (MRV). Through MRV, we are able to describe the extremal dependence structure between the layer-wise degrees for a given node; see \cite{resnick2007heavy, kulik2020heavy} for more details on MRV. For single-layer inhomogeneous random graphs, a latent weight is associated to each node that confers its connectivity in the network. \cite{bhattacharjee2022large} prove that, for a large class of inhomogeneous random graphs, if the weight distribution is univariate reguarly varying, then so is the typical degree distribution. We extend this correspondence to the multilayer setting. Here, each node is associated a weight \textit{vector} describing its connectivity in each layer of the network. We show that if the distribution of the weight vectors are multivariate regularly varying, then so is the multivariate \textit{asymptotic} degree distribution for a specified class of Norros-Reittu and Chung-Lu models. Furthermore, if the weight vector distribution has hidden regular variation (HRV), a second-order notion of MRV, then for cases where the HRV index falls within a certain range, the degree distribution also has the HRV property; see \cite{das2013living, das2015models, das2017hidden} for detailed discussion on HRV. In essence, for large networks, the extremal dependence structure exhibited by the weights may also be shared by the layer-wise degrees under certain circumstances. 

As alluded to previously, our approach differs in two key ways from \cite{bhattacharjee2022large}. First, we restrict our study of inhomogeneous graphs to ones that are similar to the Norros-Reittu and Chung-Lu graphs \cite{norros2006conditionally, chung2002average, chung2002connected}. In particular, we allow the layers of the MIRG to consist of both single-edge and multi-edge graphs. This is an especially important consideration; most multilayer networks of interest consist of both single-edge and multi-edge layers. Returning to our X example, a user may only follow another user one time, while they may retweet or like another user multiple times. Hence, the follower layer is a single-edge graph, whereas the like and retweet layers allow multiple edges between two nodes. The class of Norros-Reittu and Chung-Lu graphs considered herein also extend the single layer results in \cite{bhattacharjee2022large} to graphs such as the generalized random graph \cite{britton2006generating}. Secondly, in order to consider this extended class of Norros-Reittu and Chung-Lu graphs, we study multivariate regular variation of the \textit{asymptotic} degree distribution rather than the typical degree distribution. Modern problems in network science typically involve analyzing networks with millions or billions of nodes, hence the study of the limiting degree distribution should accurately describe the extremal behavior of observed networks \cite{newman2003structure}. 

Of particular interest for scale-free networks is the estimation of the power-law tail index. In addition to being an important object of study on its own, inference on the tail index can further inform more complex statistical procedures for scale-free networks \cite{cirkovic2024modeling, cirkovic2023preferential}. Hence, we additionally consider consistency of the well-known Hill estimator for the tail index in the MIRG model. Consistency of the Hill estimator is well-established in the iid and some time series settings \cite{resnick2007heavy, kulik2020heavy}. In the network setting, however, verifying consistency of the Hill estimator is a considerably more difficult task due to the degree dependence present in most datasets. In both preferential attachment networks and single-layer inhomogeneous random graphs, however, consistency of the Hill estimator has been shown \cite{wang2019consistency, bhattacharjee2022large}. In both cases, consistency is obtained by approximating the degree sequence by a sequence of independent processes;  branching proccesses in the case of preferential attachment graphs and latent weights for inhomogeneous random graphs \cite[see][for more details]{wang2019consistency, bhattacharjee2022large}. We follow a similar, though in some cases more general approach for the MIRG model. For the Chung-Lu random graph, \cite{bhattacharjee2022large} prove consistency of the Hill estimator when the tail index is strictly larger than $2$. By imposing a reasonable restriction on the number of order statistics used in the Hill estimator, we are able to extend the consistency results in the Chung-Lu model, among others, for tail indices strictly larger than $1$. 

Section \ref{sec:model} introduces the MIRG, along with the asymptotic properties of its degrees. Section \ref{sec:rv} provides a short introduction to MRV and HRV, then proves a correspondence between MRV/HRV of the latent weight vectors and MRV/HRV of the degree vectors. The Hill estimator is introduced and its consistency in the MIRG model is considered in Section \ref{sec:hill}. Section \ref{sec:sims} provides a short simulation study exhibiting that MRV and HRV are detectable in practically-sized networks generated from the MIRG model, as well as evidence suggesting that a restriction of the number of order statistics used in the Hill estimator is necessary to achieve consistency. Section \ref{sec:conc} discusses future research questions and potential extensions raised by the theory developed herein. Finally, in Section \ref{sec:proofs} we introduce tools from extreme value theory and random networks used to develop the main results along with their proofs. 

\subsection{Notation}

For $n \in \mathbb{N}$, we use the notation $[n]$ to denote the collection $\{1, 2, \dots, n \}$. Vectors will be denoted in boldface. If $\mathbf{z}$ is a random vector in $\mathbb{R}^p$, we let the math-italics $z_i$ denote the value of its $i$-th entry, $i \in [p]$. The zero vector in $\mathbb{R}^p$ is denoted as $\mathbf{0}$ regardless of the dimension $p$. Let $\mathbf{x}_1, \mathbf{x}_2, \dots, \mathbf{x}_n$ be vectors and $\mathcal{I} = \{i_1, i_2, \dots, i_r\} \subset [n]$, $r \leq n$. The notation $\mathbf{x}_{\mathcal{I}}$ is then used to denote the collection of vectors $\{\mathbf{x}_{i_1}, \mathbf{x}_{i_2}, \dots, \mathbf{x}_{i_r}\}$. For real numbers $a$ and $b$, let $a \wedge b$ and $a \vee b$ denote the minimum and maximum of $a$ and $b$, respectively.

\subsection{Multilayer inhomogeneous random graph model}\label{sec:model}

In this section we introduce the multilayer inhomogeneous random graph (MIRG) model. Suppose $\mathbf{A}(n) = \{A_{ijl} \}_{ijl}$ is an $n \times n \times L$ adjacency cube of $n$ nodes and $L$ layers. That is, $A_{ijl}$ represents the number of edges between nodes $i$ and $j$ in layer $l$. For every $l \in [L]$, $\{A_{ijl}\}_{ij}$ is a symmetric matrix; each layer of the MIRG is an undirected (multi)graph that allows self-loops. We let $\mathcal{L}_1 \subset [L]$ collect the layers of $\mathbf{A}(n)$ that are multigraphs (i.e. allow multiple edges) and let $\mathcal{L}_2 = [L] \setminus \mathcal{L}_1$ refer to the layers which are graphs (i.e. allow at most one edge between two nodes).

Let $\mathbf{W}_{[n]} = \{ \mathbf{W}_i \}_{i = 1}^n$ be i.i.d. random vectors in $\mathbb{R}_+^L$ with a common continuous distribution. We often refer to the $\mathbf{W}_{[n]}$ as weight vectors. We may think of $W_{il}$ as some latent attribute that conveys the connectivity of node $i$ in layer $l$. In most applications, a larger value of $W_{il}$ indicates that node $i$ will attract more edges in layer $l$, though the MIRG model permits one to augment the extent to which that is true. For each layer $l \in [L]$, we let $T_l(n) = \sum_{i = 1}^n W_{il}$ denote the sum of the weights in layer $l$.

Conditional on the latent weight vectors $\mathbf{W}_{[n]}$, the MIRG models the multigraph and graph layers similarly. For $l \in \mathcal{L}_1$, we assume
\begin{align}
\label{eq:NR}
 A_{ijl} \mid \mathbf{W}_{[n]} \overset{\text{ind}}{\sim} \text{Poisson}\left(g_l(W_{il}W_{jl}/{T_l(n)}) \right) \qquad \text{for } 1 \leq i \leq j \leq n,
\end{align}
and for $l \in \mathcal{L}_2$ we have
\begin{align}
\label{eq:Ber}
A_{ijl} \mid \mathbf{W}_{[n]} \overset{\text{ind}}{\sim} \text{Bernoulli}\left(g_l(W_{il}W_{jl}/{T_l(n)}) \right) \qquad \text{for } 1 \leq i \leq j \leq n,
\end{align}
Here, $g_l: \mathbb{R}^+ \rightarrow \mathbb{R}^+$ allows for heterogeneity in the layer-wise connection probabilities. For $l \in \mathcal{L}_2$, we naturally restrict the range of $g_l$ to the interval $[0, 1]$. Typically $g_l$ is chosen to satisfy $g_l(0) = 0$. Practical choices of $g_l$ are often monotonically increasing and concave, indicating that a larger weight contributes to more connections, although with diminishing returns.  When $g_l$ is the identity function for $l \in \mathcal{L}_1$, the multigraph layers are easily recognized as the Norros-Reittu (NR) random multigraph \cite{norros2006conditionally}. For $l \in \mathcal{L}_2$, $g_l(x) = x \wedge 1$ corresponds to the Chung-Lu random graph and $g_l(x) = x/(1 + x)$ corresponds to a version of the generalized random graph (GRG) \cite{britton2006generating, chung2002average, chung2002connected, van2008universality}. The theory presented herein covers these choices of $g_l$, and more general models as well.

Define the degree of node $i$ in layer $l$ to be $D_{il}(n) = \sum_{j = 1}^n A_{ijl}$. From \eqref{eq:NR} and the conditional independence of the $A_{ijl}$, it is easily seen that for $l \in \mathcal{L}_1$
\begin{align}
\label{eq:degreecondNR}
D_{il}(n) \mid \mathbf{W}_{[n]} \sim \text{Poisson}\left(\sum_{j = 1}^n g_l(W_{il}W_{jl}/{T_l(n)})\right), \qquad \text{for }  i = 1,\dots, n,
\end{align}
and for $l \in \mathcal{L}_2$
\begin{align}
\label{eq:degreecondBer}
D_{il}(n) \mid \mathbf{W}_{[n]} \sim \text{Poisson-binomial}\left(g_l(W_{il}W_{1l}/{T_l(n)}), \dots, g_l(W_{il}W_{nl}/{T_l(n)})\right),  
\end{align}
for $i = 1,\dots, n$ where the Poisson-binomial distribution is classically defined as the sum of independent Bernoulli random variables \cite{tang2023poisson}. For understanding properties of the degree distribution, it is far easier to work with \eqref{eq:degreecondNR} due to the availability of results regarding the Poisson distribution. Results derived from \eqref{eq:degreecondNR} are then typically translated to results for \eqref{eq:degreecondBer} using the maximal coupling of Poisson and Bernoulli random variables \cite[see Section 2.2 of][]{van2017random}. We follow a similar strategy, though in some cases we must work with \eqref{eq:degreecondBer} alone. 

We now introduce the only non-trivial condition on the functions $\{g_l\}_{l \in [L]}$ in the MIRG model. A similar assumption is made in \cite{van2008universality} when studying the distances in finite-variance inohomogeneous random graphs. Throughout, we denote the three main conditions we build upon as (C1), (C2) and (C3), the latter two of which will be introduced as needed.
\begin{itemize}
\item[\textbf{(C1)}] Each $g_l$ satisfies
\begin{align*}
g_l(x) = c_lx + O(x^{1 + \nu}) \qquad \text{as } x \downarrow 0, \text{ for some } \nu > 0,
\end{align*}
for constants $c_l > 0$.
\end{itemize}
From (C1), we may find constants $M, \delta > 0$, not depending on $l$ such that 
\begin{align}
\label{eq:glin}
|g_l(x) - c_lx| \leq M x^{1 + \nu} \qquad \text{for all } x < \delta \text{ and } l \in [L],
\end{align}
a fact we repeatedly use. Condition (C1) states for every $l \in [L]$, $g_l(x)$ can be approximated by a scalar multiple of $x$ when $x$ is small. This, combined with a finite mean assumption on the weights, allows us to approximate
\begin{align}
\label{eq:degapprox}
\sum_{j = 1}^ng_l(W_{il}W_{1l}/{T_l(n)}) \approx \sum_{j = 1}^n c_l W_{il}W_{1l}/{T_l(n)} = c_l W_{il},
\end{align}
which combined with \eqref{eq:degreecondNR} suggests that for large networks, we may be able to approximate the degree of node $i$ in layer $l$ by its weight. Furthermore, \eqref{eq:degapprox} suggests that the degrees may be approximately independent, a key ingredient in proving the consistency of the Hill estimator in the MIRG model. 

We now describe the behavior of the degree distribution in the MIRG model. Let $\mathbf{D}_i(n)$ be a random vector in $\mathbb{R}^{+}_L$ collecting the degrees of node $i$ in each layer. From \eqref{eq:degreecondNR} and \eqref{eq:degreecondBer}, it is seen that the degrees across nodes are identically distributed. Thus, throughout we only consider the distribution of $\mathbf{D}_1(n)$. The asymptotic distribution of $\mathbf{D}_1(n)$ is presented in Lemma \ref{lem:typdeg}. In order to adequately describe the asymptotic distribution, we define the random variable $\boldsymbol{\mathcal{D}}_1 \in \mathbb{R}_+^L$ which, when conditioned on $\mathbf{W}_1$, has conditionally independent Poisson entries with element-wise means $(c_1W_{11}, \dots, c_LW_{1L})$. Results  similar to Lemma \ref{lem:typdeg} have been proved in the single-layer network case (see Theorem 3.13 of \cite{bollobas2007phase} or Corollary 6.11 of \cite{van2017random}, for example), though we present the result and its proof for completeness. Lemma \ref{lem:typdeg} is proven in Section \ref{sec:typdegproof}.

\begin{lemma}
\label{lem:typdeg}
Suppose $\mathbb{E}[\|\mathbf{W}_1\|_1] < \infty$ and $\mathbf{D}_1(n)$ is generated from a MIRG satisfying (C1). Then for each $\mathbf{m} \in \mathbb{Z}_+^L$
\begin{align*}
\mathbb{P}\left(\mathbf{D}_1(n) = \mathbf{m} \right) \rightarrow \mathbb{P}\left(\boldsymbol{\mathcal{D}}_1 = \mathbf{m} \right) \equiv \mathbb{E}\left[ \prod_{l = 1}^L \frac{\left(c_lW_{1l}\right)^{m_l}e^{-c_lW_{1l}}}{m_l!} \right],
\end{align*}
as $n \rightarrow \infty$.
\end{lemma}

\section{Regular variation of the MIRG model}\label{sec:rv}

\subsection{Preliminaries on MRV and HRV}

In this section we provide a short overview of the multivariate regular variation. These are important tools to characterize the extremal dependence structure between elements of a random vector. Suppose we are working in the metric space $(\mathbb{R}_+^L, d_p)$, $L \geq 1$, where $d_p(\mathbf{x}, \mathbf{y}) = \| \mathbf{x} - \mathbf{y} \|_p$ is the $\ell_p$ norm for any $p\geq 1$ and $\mathbf{x}, \mathbf{y} \in \mathbb{R}_+^L$.  Let $\mathbb{C}_0 \subset \mathbb{C}$, be two closed cones in $\mathbb{R}_+^L$. The theoretical foundation of regularly varying measures is given by $\mathbb{M}$-convergence on $\mathbb{C}\setminus \mathbb{C}_0$, which we provide in Definition \ref{def:Mconv} \citep[see][for more on regular variation]{basrak2019note, hult2006regular, kulik2020heavy, lindskog2014regularly}. 

\begin{definition}
\label{def:Mconv}
Let $\mathbb{M}(\mathbb{C} \setminus \mathbb{C}_0)$ be the set of Borel measures on $\mathbb{C} \setminus \mathbb{C}_0$ which are finite on sets bounded away from $\mathbb{C}_0$. Let $\mathcal{C}(\mathbb{C} \setminus \mathbb{C}_0)$ denote the set of all continuous, bounded and non-negative functions on $\mathbb{C} \setminus \mathbb{C}_0$ whose supports are bounded away from $\mathbb{C}_0$. Then for $\mu_n, \mu \in \mathbb{M}(\mathbb{C} \setminus \mathbb{C}_0)$, we say that that $\mu_n \rightarrow \mu$ in $\mathbb{M}(\mathbb{C}\setminus \mathbb{C}_0)$ if
\begin{align*}
\int f d\mu_n \rightarrow \int f d\mu,
\end{align*}
for all $f \in \mathcal{C}(\mathbb{C} \setminus \mathbb{C}_0)$.
\end{definition}
Without loss of generality, we may take the functions in $\mathcal{C}(\mathbb{C} \setminus \mathbb{C}_0)$ to be uniformly continuous as well. Here, $f$ is uniformly  continuous if the modulus of continuity
\begin{align*}
\Delta_{f}^p(\delta) = \sup_{\mathbf{x}, \mathbf{y} \in \mathbb{C} \setminus \mathbb{C}_0} \left\lbrace \left|f(\mathbf{x}) - f(\mathbf{y}) \right|: \|\mathbf{x} - \mathbf{y} \|_p < \delta \right\rbrace,
\end{align*}
is such that $\Delta_{f}^p(\delta) \rightarrow 0$ as $\delta \downarrow 0$. Note that Definition \ref{def:Mconv} is equivalent to requiring that $\lim_{n \rightarrow \infty} \mu_n(A) = \mu(A)$ for all $\mu$-continuity Borel sets $A$ bounded away from $\mathbb{C}_0$ \citep[see][Theorem 2.1]{lindskog2014regularly}.  With $\mathbb{M}$-convergence in hand, we may now formally define a regularly varying distribution function in the case of $\mathbb{C} = \mathbb{R}_+^d$ and $\mathbb{C}_0 = \{\mathbf{0} \}$. 

\begin{definition}
\label{def:MRV}
We say that the distribution $\mathbb{P}\left(\mathbf{Z} \in \cdot \right)$ of a random vector $\mathbf{Z}$ on $\mathbb{R}_+^L$, $L \geq 1$, is (standard) regularly varying on $\mathbb{R}_+^L\setminus\{ \mathbf{0}\}$ with index $\alpha > 0$ if there exists some regularly varying scaling function $b(t)$ with index $1/\alpha$ and a limit measure $\nu \in \mathbb{M}(\mathbb{R}_+^L\setminus\{ \mathbf{0}\})$ such that as $t \rightarrow \infty$
\begin{align}
t\mathbb{P}\left( \mathbf{Z}/b(t) \in \cdot \right) \rightarrow \nu(\cdot), \qquad \text{in } \mathbb{M}(\mathbb{R}_+^L\setminus\{ \mathbf{0}\}).
\end{align}
If $\mathbb{P}\left(\mathbf{Z} \in \cdot \right)$ is regularly varying, we may write $\mathbb{P}\left(\mathbf{Z} \in \cdot \right) \in \text{MRV}(\alpha, b(t), \nu, \mathbb{R}_+^L\setminus\{ \mathbf{0}\})$.
\end{definition}

In applications, extreme values often occur in more than one risk region. Moreover, these risk regions may exhibit subtle behavior that may be overlooked when estimating tail probabilities by assuming just one degree of regular variation.  When the limit measure $\nu$ concentrates on a subcone of $\mathbb{R}_+^L\setminus\{ \mathbf{0}\}$, we may seek additional regular variation outside of the subcone by employing hidden regular variation as in Definition \ref{def:HRV}.

\begin{definition}
\label{def:HRV}
We say that the regularly varying distribution $\mathbb{P}\left(\mathbf{Z} \in \cdot \right)$ on $\mathbb{R}_+^L \setminus \{\mathbf{0} \}$ has hidden regular variation on $\mathbb{R}_+^L \setminus \mathbb{C}_0$ is there exists $0 < \alpha \leq \alpha_0$, scaling functions $b(t) \in \text{RV}_{1/\alpha}$, $b_0(t) \in \text{RV}_{1/\alpha_0}$ with  $b(t)/b_0(t) \rightarrow \infty$ and limit measures $\nu$, $\nu_0$ such that 
\begin{align}
\label{eq:HRV}
\mathbb{P}\left( \mathbf{Z} \in \cdot \right) \in \text{MRV}(\alpha, b(t), \nu, \mathbb{R}_+^L\setminus\{ \mathbf{0}\}) \cap \text{MRV}(\alpha_0, b_0(t), \nu_0, \mathbb{R}_+^L\setminus\mathbb{C}_0),
\end{align}
or, in other words, $\mathbb{P}\left( \mathbf{Z} \in \cdot \right) \in \text{MRV}(\alpha, b(t), \nu, \mathbb{R}_+^L\setminus\{ \mathbf{0}\})$ and $\mathbb{P}\left( \mathbf{Z} \in \cdot \right) \in \text{MRV}(\alpha_0, b_0(t), \nu_0, \mathbb{R}_+^L\setminus\mathbb{C}_0)$.
\end{definition}

For purposes of statistical inference, it is often convenient to express the limit measure $\nu_0$ in terms of generalized polar coordinates \cite{das2013living, lindskog2014regularly}. For the unit sphere with respect to the forbidden zone $\mathbb{C}_0$ defined by $\aleph_{\mathbb{C}_0} = \left\lbrace \mathbf{x} \in   \mathbb{R}_+^L\setminus\mathbb{C}_0 : d_p(\mathbf{x}, \mathbb{C}_0) = 1 \right\rbrace$, we may further define the generalized polar coordinate transform GPOLAR:  $\mathbb{R}_+^L\setminus\mathbb{C}_0 \mapsto (0, \infty) \times \aleph_{\mathbb{C}_0}$ by
\begin{align*}
\text{GPOLAR}(\mathbf{x}) = \left(d_p(\mathbf{x}, \mathbb{C}_0), \frac{\mathbf{x}}{d_p(\mathbf{x}, \mathbb{C}_0)} \right).
\end{align*}
With the GPOLAR transformation in hand, we may rewrite \eqref{eq:HRV} as
\begin{align}
\label{eq:GPOLAR}
t\mathbb{P}\left( \left(\frac{d_p(\mathbf{Z}, \mathbb{C}_0)}{b_0(t)}, \frac{\mathbf{Z}}{d_p(\mathbf{Z}, \mathbb{C}_0)}\right) \in \cdot \right) \rightarrow (\nu_{\alpha_0} \times S_0)(\cdot) \qquad \text{in } \mathbb{M}((0, \infty) \times \aleph_{\mathbb{C}_0}),
\end{align}
where $\nu_{\alpha_0}(x, \infty) = x^{-\alpha_0}$ for $x > 0$ and $S_0$ is a probability measure on $\aleph_{\mathbb{C}_0}$. The convergence \eqref{eq:GPOLAR} immediately implies that $d_p(\mathbf{Z}, \mathbb{C}_0)$ has a regularly varying distribution function with tail index $\alpha_0$. 

\subsection{MRV and HRV in the MIRG model}\label{sec:MRVMIRG}

In this section we relate regular variation of the weights to regular variation of the degrees in the MIRG model. For a large class of inhomogeneous random graphs (i.e. single-layer), \cite{bhattacharjee2022large} find that regular variation of the weight distribution implies regular variation of the typical degree distribution. It is important to note that this is a finite-sample result; regular variation of the typical degree distribution is achieved irrespective of the network size. Modern network science, however, often deals with networks that have an extremely large number of vertices \cite{newman2003structure}. Thus, it is also reasonable to instead consider properties of the asymptotic degree distribution. Inspection of the asymptotic degree distribution also permits the study of a range of models beyond the Chung-Lu or Norros-Reittu random graphs as mentioned in Section \ref{sec:model}.

As elucidated in Theorem \ref{thm:rv}, we find that multivariate regular variation of the weight distribution implies multivariate regular variation of the asymptotic degree distribution. Furthermore, the asymptotic degree distribution inherits the same scaling function and limit measure as the (scaled) weight distribution. Here, we define $\boldsymbol{\mathcal{W}}_1 = (\mathcal{W}_{11}, \dots, \mathcal{W}_{1L}) \equiv (c_1W_{11}, \dots, c_LW_{1L})$ and $\boldsymbol{\mathcal{D}}_1$ as in Lemma \ref{lem:typdeg}. In addition, from Definition \ref{def:HRV}, it is clear that hidden regular variation of the weight distribution also implies hidden regular variation of the degree distribution, as long as the hidden regular variation is not too subtle. Theorem \ref{thm:rv} is proven in Section \ref{sec:rvproof}.

\begin{theorem}
\label{thm:rv}
Let $\alpha > 0$ and $\alpha_0 \in [\alpha, 2\alpha)$. 
\begin{itemize}
\item[(a)] If $\mathbb{P}\left(\boldsymbol{\mathcal{W}}_1 \in \cdot \right) \in \text{MRV}(\alpha, b(t), \nu, \mathbb{R}_+^L\setminus \{\mathbf{0}\})$, for some scaling function $b(t) \in RV_{1/\alpha}$ and limit measure $\nu$, then $\mathbb{P}\left(\boldsymbol{\mathcal{D}}_1 \in \cdot \right) \in \text{MRV}(\alpha, b(t), \nu, \mathbb{R}_+^L\setminus \{\mathbf{0}\})$.
\item[(b)] If $\mathbb{P}\left(\boldsymbol{\mathcal{W}}_1 \in \cdot \right) \in \text{MRV}(\alpha, b(t), \nu, \mathbb{R}_+^L\setminus  \{\mathbf{0}\})\cap\text{MRV}(\alpha_0, b_0(t), \nu_0, \mathbb{R}_+^L\setminus \mathbb{C}_0)$, for scaling functions $b(t) \in RV_{1/\alpha}$, $b_0(t) \in RV_{1/\alpha_0}$ with $b(t)/b_0(t) \rightarrow \infty$ as $t \rightarrow \infty$, limit measures $\nu$, $\nu_0$, and a closed cone $\mathbb{C}_0$, then $\mathbb{P}\left(\boldsymbol{\mathcal{D}}_1 \in \cdot \right) \in \text{MRV}(\alpha, b(t), \nu, \mathbb{R}_+^L\setminus \{\mathbf{0}\})\cap\text{MRV}(\alpha_0, b_0(t), \nu_0, \mathbb{R}_+^L\setminus \mathbb{C}_0)$.
\end{itemize}
\end{theorem}

For part (b), note that we require $\alpha_0 \in [\alpha, 2\alpha)$. Seemingly, even if regular variation of $\mathbb{P}\left(\boldsymbol{\mathcal{W}}_1 \in \cdot \right)$ is confined to a subcone $\mathbb{C}_0$, $\mathbb{P}\left(\boldsymbol{\mathcal{D}}_1 \in \cdot \right)$ will produce lighter tailed behavior on  $\mathbb{R}_+^L\setminus \ \mathbb{C}_0$ due to the Poisson dispersion of $\boldsymbol{\mathcal{D}}_1$ off of $\mathbb{C}_0$. 
As evidenced in simple example to follow, the behavior off of $\mathbb{C}_0$ is not necessarily hidden regular variation. The example also shows that degree distances from $\mathbb{C}_0$, however, can be regularly varying with tail index $2\alpha$. Thus, in order for hidden regular variation of $\mathbb{P}\left(\boldsymbol{\mathcal{W}}_1 \in \cdot \right)$ to confer to $\mathbb{P}\left(\boldsymbol{\mathcal{D}}_1 \in \cdot \right)$, we require that the hidden tail behavior of  $\mathbb{P}\left(\boldsymbol{\mathcal{W}}_1 \in \cdot \right)$ dominates the Poisson dispersion. Otherwise, the hidden regular variation from $\mathbb{P}\left(\boldsymbol{\mathcal{W}}_1 \in \cdot \right)$ may be corrupted by degrees that spread from $\mathbb{C}_0$.

\begin{example}
\label{ex:example}
To explore the tail behavior of $\mathbb{P}\left(\boldsymbol{\mathcal{D}}_1 \in \cdot \right)$ on $\mathbb{R}_+^L\setminus \ \mathbb{C}_0$, we provide a simple example in $(\mathbb{R}_+^2, d_2)$. Suppose $W_{i1}$ is regularly varying with tail index $\alpha > 0$ and let $W_{i2} = W_{i1}$ for $i \in [n]$. It is easily seen that $\mathbf{W}_1$ is multivariate regularly varying and satisfies full asymptotic dependence \cite{das2017hidden}. That is, $\mathbb{P}\left(\boldsymbol{\mathcal{W}}_1 \in \cdot \right) \in \text{MRV}(\alpha, b(t), \nu, \mathbb{R}_+^L\setminus  \{\mathbf{0}\})$ and the limit measure $\nu$ only place mass on the ray $\mathbb{C}_0 = \{\mathbf{x} \in \mathbb{R}_+^2: x_1 = x_2 \}$. Here we take $b(t)$ to be the $1-1/t$ quantile of $W_{11}$, $t \geq 1$.  Let 
\begin{align*}
A_{ij1} \mid \mathbf{W}_{[n]} \overset{\text{ind}}{\sim}& \text{Poisson}(W_{i1}W_{j1}/T_1(n)) \qquad \text{for } 1 \leq i \leq j \leq n,\\
A_{ij2} \mid \mathbf{W}_{[n]} \overset{\text{ind}}{\sim}& \text{Poisson}(W_{i2}W_{j2}/T_2(n)) \qquad \text{for } 1 \leq i \leq j \leq n,
\end{align*}
so that $\mathbf{D}_1(n)$ and $\boldsymbol{\mathcal{D}}_1$ have the same conditional distribution given $\mathbf{W}_{[n]}$, for all $n \in \mathbb{N}$. In order to determine whether hidden regular variation exists off of $\mathbb{C}_0$, we evaluate the joint distribution of $d_2(\boldsymbol{\mathcal{D}}_1, \mathbb{C}_0) = |\mathcal{D}_{11} - \mathcal{D}_{12} |/\sqrt{2}$ and $\boldsymbol{\mathcal{D}}_1/d_2(\boldsymbol{\mathcal{D}}_1, \mathbb{C}_0)$. Let $b_0(t)$ be the $1-1/t$ quantile of $\sqrt{W_{11}}$. In Section \ref{sec:example}, we prove that for any $u, v \geq 0$
\begin{align}
\label{eq:example_final}
\begin{split}
t\mathbb{P}\left( \frac{|\mathcal{D}_{11} - \mathcal{D}_{12}|}{b_0(t)} \left(\frac{\sqrt{\pi}}{\Gamma(\alpha + \frac{1}{2})} \right)^{\frac{1}{2\alpha}} > u, \frac{\sqrt{2}\mathcal{D}_{11}/\sqrt{W_{11}}}{|\mathcal{D}_{11} - \mathcal{D}_{12}|} > v\right)& \\
\rightarrow u^{-2\alpha} \cdot \frac{2^{1 - \alpha}\sqrt{\pi}}{\Gamma(\alpha + \frac{1}{2})}&\int_0^{1/v} z^{2\alpha} \phi(z) dz,
\end{split}
\end{align}
where $\phi(\cdot)$ is the standard normal density function. Statement \eqref{eq:example_final} immediately gives that the marginal distribution of $d_2(\boldsymbol{\mathcal{D}}_1, \mathbb{C}_0)$ is regularly varying with tail index $2\alpha$. Thus, as a practical concern, one is unable to detect potential hidden regular variation conferred from $\mathbb{P}\left(\boldsymbol{\mathcal{W}}_1 \in \cdot \right)$ unless the associated tail index, $\alpha_0$, is smaller than $2\alpha$ (see Section \ref{sec:hrvsim} for more details). In addition, \eqref{eq:example_final} suggests that $\mathbb{P}\left(\boldsymbol{\mathcal{D}}_1 \in \cdot \right)$ is not hidden regularly varying off of $\mathbb{C}_0$. In essence, the lack of hidden regular variation is due to a mismatch of orders. Conditional on large $W_{11}$, the distance of $\boldsymbol{\mathcal{D}}_1$ from the diagonal is on the order of $\sqrt{W_{11}}$ while its component-wise means are both $W_{11}$. Hence, in order for the distribution of $\boldsymbol{\mathcal{D}}_1/d_2(\boldsymbol{\mathcal{D}}_1, \mathbb{C}_0)$ to stabilize and result in a limiting product measure, one must correct this mismatch by scaling the distances by $\sqrt{W_{11}}$. 
\end{example}

Nevertheless, Theorem \ref{thm:rv} provides a flexible framework for generating multilayer networks with multivariate regularly varying degree distributions. Furthermore, full control is exerted on the extremal dependence demonstrated by the degrees. This is in stark contrast with the currently available avenues of modeling degree dependence in network science. For example, in single-layer directed preferential attachment networks, the state of the art only allows extreme out/in-degrees to concentrate on lines through the origin \cite{wang2023random}. Compared to the direct simulation of networks with multivariate regularly varying degree distributions, generating weights in $\mathbb{R}_+^L$ with a given extremal dependence structure is a much simpler task \citep[see][for more on generating data with MRV]{das2015models}. Additionally, said dependence structure is immediately conferred to the degrees. Such properties, along with the flexibility allotted between layers in both edge counts and node affinity make the MIRG model a formidable choice for modeling multilayer networks with extremal degree behavior.

\section{Tail index estimation in the MIRG model}\label{sec:hill}

\subsection{Preliminaries on the Hill estimator}\label{sec:hillback}

In this section, we formally introduce the Hill estimator and provide definitions of associated tools commonly used to prove its consistency. Suppose $X_1, X_2, \dots, X_n$ is a sequence of iid regularly varying random variables with tail index $\alpha > 0$. In addition, let their order statistics be given by $X_{(1)}\geq X_{(2)} \geq \dots  \geq X_{(n)}$. Then, the Hill estimator of the inverse of the tail index is given by
\begin{align}
\label{eq:hillestdef}
H_{k, n} =  \frac{1}{k}\sum_{i = 1}^k \log \frac{X_{(i)}}{X_{(k + 1)}},
\end{align}
where $k \geq 1$ denotes the number of order statistics used in in the estimation of $\alpha^{-1}$ \cite{hill1975simple}. When $k = k_n \rightarrow \infty$ and $k_n/n \rightarrow \infty$ as $n \rightarrow \infty$, we have that 
\begin{align*}
H^{-1}_{k, n} \xrightarrow{p} \alpha.
\end{align*}
Note that the Hill estimator has been shown to be consistent in a wide variety of data contexts beyond the iid setting. Such instances include stationary time series and network datasets \cite{mason1982laws, resnick1995consistency, wang2019consistency}. For iid data, a minmum distance procedure has been proposed to choose the optimal number of order statistics used in the estimation of $\alpha$ \cite{clauset2009power, drees2020minimum}.  In the network setting, the Hill estimator has received considerable attention in applications due to the effect the tail index has on other properties of the degree distribution \cite{dorogovtsev2008critical}. \cite{voitalov2019scale} use the Hill estimator, among other tail index estimators, to verify that a significant portion of networks have power law degree distributions. In addition, \cite{voitalov2019scale} discuss the practical considerations when employing the Hill estimator to network data, including the difficulty faced in threshold selection in the presence of discrete power-law data. 

Often, consistency of the Hill estimator is derived through analysis of the tail empirical measure
\begin{align}
\label{eq:tailempdef}
\nu_n(\cdot) = \frac{1}{k_n} \sum_{i = 1}^n \epsilon_{X_i/b(n/k_n)}(\cdot),
\end{align}
where $b(t)$ satisfies $\mathbb{P}\left(X_i > b(t) \right) \sim 1/t$ as $t \rightarrow \infty$. Here, $\epsilon_x(A)$ places point mass on $A \subset (0, \infty]$ if $x \in A$ and is $0$ otherwise. Denote the set of non-negative Radon measures on $(0, \infty]$ by $M_+((0, \infty])$ and define the measure $\nu_\alpha$ by $\nu_\alpha(y, \infty] = y^{-\alpha}$ for $y > 0$. By seeking weak convergence of the random measure $\nu_n$ to $\nu_\alpha$ in $M_+((0, \infty])$, one may derive consistency of the Hill estimator through continuous mapping arguments. Such methods are reviewed in Section \ref{sec:hillest} in the MIRG model setting. See Theorem \ref{thm:weakconvD} for more details.

\subsection{Hill estimation in the MIRG model}

In this section we prove consistency of the Hill estimator in the MIRG model when the weight vectors are multivariate regularly varying with tail index $\alpha > 1$. Consistency of the Hill estimator in the single layer case is proven for the Norros-Reittu model for $\alpha > 0$ and Chung-Lu model for $\alpha > 2$ in \cite{bhattacharjee2022large}. Though one could consider applying the Hill estimator to each individual layer-wise degree sequence $\{D_{il}(n) \}_{i = 1}^n$, $l \in [L]$, in the MIRG model we consider a procedure that simultaneously employs all of $\{\mathbf{D}_i(n)\}_{i = 1}^n$ in the estimation, i.e. $\|\mathbf{D}_i(n)\|_p$. 
This is also a common strategy when exploring the dependence structure of multivariate extremes.
In addition, we extend consistency of the Hill estimator to the case of $\alpha \in (1, 2]$ for the Chung-Lu model, a case not covered in \cite{bhattacharjee2022large}, among other inhomogenous random graphs. The case where $\alpha \in (1, 2]$ is of distinct importance since many networks of interest have a tail index within that range \cite{albert2002statistical, albert1999diameter, barabasi2000scale, newman2003structure}.   

Thus, in order to meaningfully estimate $\alpha$ from the degrees of the MIRG model, we introduce our second assumption:
\begin{itemize}
\item[\textbf{(C2)}] The random vector $\boldsymbol{\mathcal{W}}_1$ is such that $\mathbb{P}\left(\boldsymbol{\mathcal{W}}_1 \in \cdot \right) \in \text{MRV}(\alpha, b(t), \nu, \mathbb{R}_+^L\setminus\{ \mathbf{0}\})$ for $\alpha > 1$.
\end{itemize}
If the weights $\mathbf{W}_{[n]}$ were accessible, condition (C2) implies the tail index $\alpha$ could be estimated by using Hill estimator computed on the weight radii $\{ \| \boldsymbol{\mathcal{W}}_i \|_p \}_{i \in [n]}$ for some $p \in \mathbb{N}$. That is, one would estimate $\alpha^{-1}$ via the statistic
\begin{align*}
H^\star_{k, n} = \frac{1}{k}\sum_{i = 1}^k \log \frac{R^\star_{(i)}(n)}{R^\star_{(k + 1)}(n)},
\end{align*}
where we define $R^\star_i(n) = \| \boldsymbol{\mathcal{W}}_i \|_p$ for $i \in [n]$. Since the weights are inaccessible, however, it behooves us to instead consider employing the estimator based on the degrees
\begin{align*}
H_{k, n} = \frac{1}{k}\sum_{i = 1}^k \log \frac{R_{(i)}(n)}{R_{(k + 1)}(n)},
\end{align*}
where $R_i(n) = \| \mathbf{D}_i(n) \|_p$. From, \eqref{eq:degreecondNR}, \eqref{eq:degreecondBer} and the approximation \eqref{eq:degapprox}, we may expect that $R^\star_{(i)}$ is close to $R_{(i)}$ for $i \in [n]$. This approximation is more viable when $i$ is small since we expect strong concentration of $R_{(i)}$ around $R^\star_{(i)}$ given the mixed-Poisson nature of the asymptotic degree distribution in Lemma \ref{lem:typdeg} and the fact that Poisson random variables with large rates experience strong concentration. Hence, we expect that $H_{k, n}$ should approximate $H^\star_{k, n}$, and therefore $H^{-1}_{k_n, n}$ should be a consistent estimator of $\alpha$ for a carefully specified sequence $k_n$.  

In order to make these notions rigorous, we introduce the tail empirical measure based on the sequence $\{ R^\star_i \}_{i = 1}^n$ 
\begin{align}
\label{eq:tailempW}
\nu^\star_n(\cdot) = \frac{1}{k_n} \sum_{i = 1}^n \epsilon_{R^\star_i(n)/b(n/k_n)}(\cdot),
\end{align} 
where we take $b(t)$ to be the $1 - 1/t$ quantile of the distribution of $R^\star_1$ for $t \geq 1$. Since $R^\star_1$ has a regularly varying distribution function with tail index $\alpha$, then $b(t)$ is reguarly varying with index $1/\alpha$ \cite[see Lemma 3.3 of][]{bhattacharjee2022large}. Under (C2), it is well known that 
\begin{align}
\label{eq:weakconvW}
\nu^\star_n \Rightarrow \nu_\alpha.
\end{align}
in $M_+((0, \infty])$ as $n \rightarrow \infty$, $k_n \rightarrow \infty$ and $k_n/n \rightarrow 0$ \cite[see Theorem 4.1 in][]{resnick2007heavy}. We seek an analogous weak convergence result for the tail empirical measure based on the degrees
\begin{align}
\label{eq:tailempD}
\nu_n(\cdot) = \frac{1}{k_n} \sum_{i = 1}^n \epsilon_{R_i(n)/b(n/k_n)}(\cdot),
\end{align} 
As in \cite{bhattacharjee2022large}, this weak convergence is obtained via an approximation by the tail empirical measure based on the independent weights, $\nu^\star_n$. Weak convergence of $\nu_n$ is presented in Theorem \ref{thm:weakconvD}. In order to prove Theorem \ref{thm:weakconvD}, we require our third and final condition.
\begin{itemize}
\item[\textbf{(C3)}] Suppose $n^{1/\alpha}/k_n = O(n^{-\kappa})$ for some $\kappa \in (0, (\alpha - 1)/\alpha)$.
\end{itemize}
Condition (C3) implies that in order to achieve the weak convergence of the tail empirical measure, and thus consistency of the Hill estimator in the MIRG model, $k_n$ must be chosen to grow at a slightly faster rate than $n^{1/\alpha}$. That is, as the tail becomes heavier, more order statistics must be employed to obtain consistency of the Hill estimator. The assumption (C3) emerges in the approximation \eqref{eq:degapprox} which is valid when the ratio of the weights to the sum of the weights are small. As $\alpha$ approaches $1$ from above, however, we may expect a portion of the large weights to be so large that their associated ratios are near $1$ and the approximation no longer holds. In order to counteract this effect, we must ensure that the Hill estimator is majorly compromised of smaller order statistics that satisfy the approximation \eqref{eq:degapprox}. 
\begin{theorem}
\label{thm:weakconvD}
Let $\{\mathbf{D}_i(n)\}_{i = 1}^n$ be a degree sequence from a MIRG satisfying (C1) and (C2). Suppose $k_n$ is an intermediate sequence satisfying $k_n \rightarrow \infty$ and $k_n/n \rightarrow 0$ as $n \rightarrow \infty$. Additionally assume that $k_n$ satisfies (C3). Then
\begin{align}
\label{eq:weakconvD}
\nu_n \Rightarrow \nu_\alpha,
\end{align}
in $M_+((0, \infty])$ as $n \rightarrow \infty$. 
\end{theorem}
Theorem \ref{thm:weakconvD} is proven in Section \ref{sec:tempapprox}. In order to obtain consistency of the Hill estimator from Theorem \ref{thm:weakconvD}, an intermediate step is to use standard continuous mapping arguments reviewed in Section \ref{sec:hillest} to derive weak convergence of 
\begin{align}
\label{eq:tailemphatD}
\hat{\nu}_n(\cdot) = \frac{1}{k_n} \sum_{i = 1}^n \epsilon_{R_i(n)/R_{(k_n)}(n)}(\cdot),
\end{align}
where the unknown $b(n/k_n)$ in \eqref{eq:tailempD} is estimated by $R_{(k_n)}(n)$. That is, we obtain 
\begin{align}
\label{eq:weakconvDhat}
\hat{\nu}_n \Rightarrow \nu_\alpha,
\end{align}
in $M_+((0, \infty])$ as $n \rightarrow \infty$. Such steps are outlined in Section \ref{sec:hillest}. With \eqref{eq:weakconvDhat} in hand, we may now present the consistency of the Hill estimator for the MIRG model.
\begin{theorem}
\label{thm:hillc}
Let $\{\mathbf{D}_i(n)\}_{i = 1}^n$ be a degree sequence from a MIRG satisfying (C1) and (C2). Suppose $k_n$ is an intermediate sequence satisfying $k_n \rightarrow \infty$ and $k_n/n \rightarrow 0$ as $n \rightarrow \infty$. Additionally assume that $k_n$ satisfies (C3). Then as $n \rightarrow \infty$
\begin{align*}
H_{k_n, n} \xrightarrow{p} 1/\alpha,
\end{align*}
in $\mathbb{R}$. 
\end{theorem}
The proof of Theorem \ref{thm:hillc} is provided in Section \ref{sec:hillest}.

\section{Simulation evidence}\label{sec:sims}

In this section, we provide some simulations that elucidate the phenomena presented in Theorems \ref{thm:rv} and \ref{thm:hillc}. In particular, for Theorem \ref{thm:rv}, we present empirical evidence supporting the claim that hidden regular variation can be detected in the MIRG model. For Theorem \ref{thm:hillc}, we perform simulations that inspect the neccessity of assumption (C3) for achieving consistency of the Hill estimator in the MIRG model. 

\subsection{Hidden regular variation in the MIRG}\label{sec:hrvsim}

In this section, we present an experiment that indicates that hidden regular variation may be detected in the MIRG model. Following a similar procedure as in \cite{das2017hidden}, Example 2, for generating hidden regular variation, suppose $V_1 \sim \text{Pareto}(\alpha)$ and $V_2 \sim \text{Pareto}(\alpha_0)$ independently, where $\alpha = 1.1$ and $\alpha_0 = 1.3$. Suppose $\Theta_1 \sim \text{Beta}(5, 5, 0.4, 0.6)$ and $\Theta_2 \sim \text{Uniform}(0, 0.4)$ independently of each other, $V_1$ and $V_2$. Here, $Y \sim \text{Beta}(b_1, b_2, c_1, c_2)$ if $Y = (c_2 - c_1)X + c_1$ for $X \sim \text{Beta}(b_1, b_2)$ and $b_1, b_2 > 0$, $c_2 > c_1 \geq 0$. Then let
\begin{align}
\label{eq:simW}
\mathbf{W} = (W_1, W_2) = 
\begin{cases} 
(V_1\Theta_1, V_1(1-\Theta_1)) & \text{with probability } 1/2 \\
(V_2\Theta_2, V_2(1-\Theta_2)) & \text{with probability } 1/2
\end{cases}.
\end{align}
In \eqref{eq:simW}, we have employed the inverse $\ell_1$ polar coordinate transform to produce a random variable $\mathbf{W}$ that possesses MRV on $\mathbb{R}^2_+\setminus \{\mathbf{0}\}$ with tail index $\alpha = 1.1$. In particular, the limit measure concentrates mass on the cone
\begin{align*}
\mathbb{C}_0 = \left\lbrace \mathbf{x} \in \mathbb{R}_+^2 : \frac{2}{3}x_1 \leq x_2 \leq \frac{3}{2}x_1 \right\rbrace. 
\end{align*} 
We may seek further regular variation off of $\mathbb{C}_0$. For simplicity, \eqref{eq:simW} was designed so that the limit measure $\nu_0$ only places mass above $\mathbb{C}_0$, though hidden regular variation below $\mathbb{C}_0$ may also be considered. Using the Euclidean metric, 
\begin{align*}
d_2(\mathbf{x}, \mathbb{C}_0) = \frac{x_2 - 1.5x_1}{\sqrt{1 + 1.5^2}}, \qquad \text{for } x_2 > 1.5x_1,
\end{align*}
and $\aleph_{\mathbb{C}_0}$ is given by
\begin{align*}
&\aleph_{>\mathbb{C}_0} \cup  \aleph_{<\mathbb{C}_0} \\
&= \left\lbrace (v, 1.5v + \sqrt{1 + 1.5^2}): v \geq 0 \right\rbrace \cup \left\lbrace (v, (2/3) v - \sqrt{1 + (2/3)^2}): v \geq 1.5\sqrt{1 + (2/3)^2} \right\rbrace.
\end{align*}
Hence, using the GPOLAR transformation, we may rewrite \eqref{eq:HRV}
\begin{align}
\label{eq:WHRV}
&t\mathbb{P}\left( \left( \frac{W_2 - 1.5W_1}{\sqrt{1 + 1.5^2} b_0(t)},  \frac{\sqrt{1 + 1.5^2} \mathbf{W}}{ W_2 - 1.5W_1}\right) \in \cdot  \right) \rightarrow (\nu_{\alpha_0} \times S_0)(\cdot) \qquad \text{in }  \mathbb{M}((0, \infty) \times \aleph_{>\mathbb{C}_0}).
\end{align}
Since $S_0$ places no mass on $\aleph_{<\mathbb{C}_0}$, we omit the trivial convergence in $\mathbb{M}((0, \infty) \times \aleph_{<\mathbb{C}_0})$. Following \cite{das2017hidden}, when analyzing the hidden regular variation in \eqref{eq:WHRV}, it suffices to consider the pair $(W_2 - 1.5W_1, W_2/W_1)$ since for $y > 1.5$
\begin{align*}
t\mathbb{P}&\left(\frac{W_2 - 1.5W_1}{b_0(t)} > x, \frac{W_2}{W_1} \leq y \right)\\
&\rightarrow \left(1 + 1.5^2\right)^{-\alpha_0/2}x^{-\alpha_0} S_0\left\lbrace (u_1, 1.5u_1 + \sqrt{1 + 1.5^2}); u_1 \geq \frac{\sqrt{1 + 1.5^2}}{y  - 1.5} \right\rbrace.
\end{align*}
We consider $n = 2{,}000{,}000$ replicates of $\mathbf{W}$, $\mathbf{W}_{[n]}$. Generate a $L = 2$ layer MIRG with adjacency cube $\mathbf{A}(n)$ as follows. Let
\begin{align}
\label{eq:simMIRG}
\begin{split}
A_{ij1} &\mid \mathbf{W}_{[n]} \overset{\text{ind}}{\sim} \text{Poisson}\left( W_{i1}W_{j1}/T_1(n) \right), \qquad \text{for } 1 \leq i \leq j \leq n, \\
A_{ij2} &\mid \mathbf{W}_{[n]} \overset{\text{ind}}{\sim} \text{Bernoulli}\left( 1 - \exp\left\lbrace -W_{i2}W_{j2}/T_2(n) \right\rbrace \right), \qquad \text{for } 1 \leq i \leq j \leq n.
\end{split}
\end{align}
Borrowing nomenclature from Section \ref{sec:model}, we have that $g_1(x) = x$ and $g_2(x) = 1 - e^{-x}$ for $x > 0$. That is, the first layer of the MIRG consists on a Norros-Reittu random graph and the second layer consists of what is commonly known as the Poissonian random graph \cite{van2008universality}. It is clear that $g_1$ and $g_2$ satisfy assumption (C1). According to Theorem \ref{thm:rv}, statement \eqref{eq:WHRV} should approximately hold with $\mathbf{D}_1(n)$ in place of $\mathbf{W}$, assuming that $n$ is large enough so that $\mathbb{P}(\mathbf{D}_1(n) \in \cdot )$ is close to $\mathbb{P}(\boldsymbol{\mathcal{D}}_1 \in \cdot )$. 

In order to detect hidden regular variation of $\mathbf{D}_1(n)$, we employ the Hillish estimator \cite{das2011detecting}. Theorem \ref{thm:rv}, along with \eqref{eq:WHRV} suggests that the pairs $(\xi_i, \eta_i) \equiv (D_{i2}(n) - 1.5D_{i1}(n), D_{i2}(n)/D_{i1}(n))$, $i \in [n]$, satisfy a conditional extreme value (CEV) model. The Hillish estimator aims to detect CEV models by defining $\xi_{(1)} \geq \dots \geq \xi_{(n)}$ and letting $\eta^\star_i$ be the $\eta$-variable corresponding to $\xi_{(i)}$, also known as the concomitant of $\eta^\star_i$. If $N_i^k$ is the rank of $\eta^\star_i$ among $\eta^\star_1, \dots, \eta^\star_k$, the Hillish estimator for $k \in [n]$ is defined as
\begin{align}
\label{eq:Hillish}
\text{Hilllish}_{k, n}((\xi, \eta)_{[n]})  := \frac{1}{k} \sum_{i = 1}^k \log\left(\frac{k}{i}\right)\log\left(\frac{k}{N_i^k}\right). 
\end{align}
It is shown in \cite{das2011detecting} that if $(\xi, \eta)_{[n]}$ are independent pairs satisfying a CEV model, then $\text{Hilllish}_{k_n, n}((\xi, \eta)_{[n]})$ has an in probability limit as $n \rightarrow \infty$ and $k_n/n \rightarrow 0$. Additionally, the limiting probability measure associated with the CEV model is a product measure if and only if
\begin{align}
\label{eq:Hillishprod}
\text{Hilllish}_{k_n, n}((\xi, \eta)_{[n]}) \xrightarrow{p} 1 \text{ and }\text{Hilllish}_{k_n, n}((\xi, -\eta)_{[n]}) \xrightarrow{p} 1,
\end{align}
as $n \rightarrow \infty, k_n/n \rightarrow 0$. 
Although the theoretical properties of the Hillish estimator have been analyzed only under the iid assumption, \cite{das2017hidden} applies it to network data to detect the existence of HRV. Hence, we follow a similar strategy here.

Since the limit in \eqref{eq:WHRV} is a product measure, we test whether $\mathbf{D}_1(n)$ has hidden regular variation by simulating $1{,}000$ replicates of the MIRG with $n = 2{,}000{,}000$ nodes and plotting the Hillish estimator for $k \in [4{,}000]$ over the $1{,}000$ replicates. Although the Hillish estimator is designed with independent data in mind, the approximation \eqref{eq:degapprox} suggests that the dependence between degrees may not overly inhibit its ability to detect HRV. The left-hand panel of Figure \ref{fig:hillish} presents pointwise $(10, 90)$-th (light pink) and $(25, 75)$-th (purple) quantiles for the Hillish estimators across the $1{,}000$ iterations. The black line represents the pointwise Hillish means across the $1{,}000$ iterations. The plots of both Hillish estimators suggest that $\mathbf{D}_1(n)$ does indeed exhibit hidden regular variation; the Hillish estimators stabilize around $1$ near $k_n = 500$. As $k_n$ becomes large however, the Hillish estimators drift away from $1$, indicating that that $k_n$ has grown too large in comparison to $n$ in order to satisfy the convergence in \eqref{eq:Hillishprod}. 

\begin{figure}
\includegraphics[width=\textwidth]{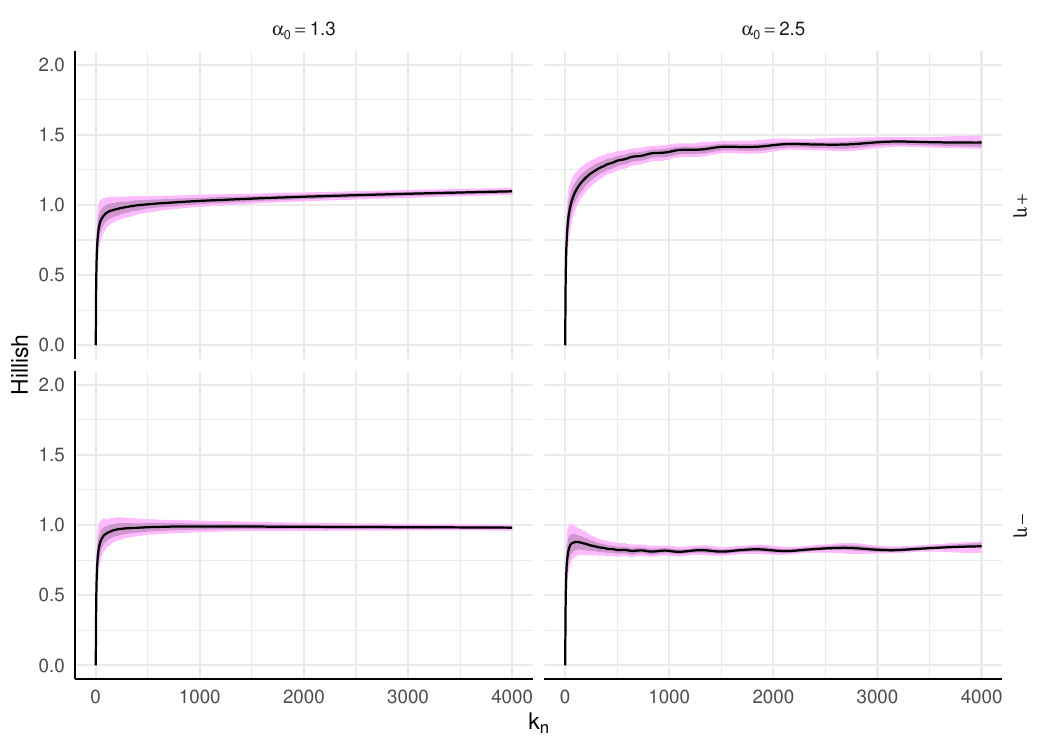}
\caption{Pointwise means (black), $(10, 90)$-th (light pink) and $(25, 75)$-th (purple) quantiles for the Hillish estimators based on $(\xi, \eta)_{[n]}$ and $(\xi, -\eta)_{[n]}$ for $n = 2{,}000{,}000$ and $k \in [4,000]$ over $1{,000}$ replicates of the MIRG model with $\alpha = 1.1$ and $\alpha_0 \in \{ 1.3, 2.5 \}$.}
\label{fig:hillish}
\end{figure}

In order to guarantee that hidden regular variation from $\boldsymbol{\mathcal{W}}_1$ confers to $\boldsymbol{\mathcal{D}}_1$, Theorem \ref{thm:rv} requires that $\alpha_0 \in [\alpha, 2 \alpha)$. To further assess the necessity of this requirement, we simulate $1{,}000$ replicates of the MIRG \eqref{eq:simMIRG} with $n = 2{,}000{,}000$ except now $\alpha_0 = 2.5$. In this setting, $\alpha_0 \notin [\alpha, 2\alpha)$ and we do not expect that hidden regular variation from the weights is detectable from the degrees as discussed in Section \ref{sec:MRVMIRG}. For each replicate, we compute the Hillish estimator over $k \in [4{,}000]$. The right-hand panel of Figure \ref{fig:hillish} plots the $(10, 90)$-th  and $(25, 75)$-th quantiles as well as pointwise means for the Hillish estimators across the $1{,}000$ iterations. The Hillish plots indicate that the observed $\mathbf{D}_1(n)$ do not possess hidden regular variation; neither set of Hillish estimators concentrate around $1$ for any value of $k \in [4{,}000]$.

\subsection{Consistency of the Hill estimator in the MIRG}

In this section we present a simulation that inspects the necessity of condition (C3) to obtain consistency of the Hill estimator in the MIRG model for $\alpha > 1$. Let 
\begin{align*}
\mathbf{W} = (W_{1}, W_{2}) = (V\Theta, V(1 - \Theta)),
\end{align*}
where $V \sim \text{Pareto}(\alpha)$ and $\Theta \sim \text{Beta}(5, 5, 0.4, 0.6)$ independently for $\alpha > 0$. By construction, $\mathbf{W}$ exhibits multivariate regular variation with tail index $\alpha$. Consider $n = 1{,}000{,}000$ independent replicates of $\mathbf{W}$, $\mathbf{W}_{[n]}$, and define a $L = 2$ layer MIRG with adjacency cube $\mathbf{A}(n)$ as follows. Let
\begin{align*}
A_{ij1} &\mid \mathbf{W}_{[n]} \overset{\text{ind}}{\sim} \text{Poisson}\left( W_{i1}W_{j1}/T_1(n) \wedge 1  \right), \qquad \text{for } 1 \leq i \leq j \leq n, \\
A_{ij2} &\mid \mathbf{W}_{[n]} \overset{\text{ind}}{\sim} \text{Bernoulli}\left( W_{i2}W_{j2}/T_2(n)/(1 + W_{i2}W_{j2}/T_2(n))  \right), \qquad \text{for } 1 \leq i \leq j \leq n.
\end{align*}
Here, $g_1(x) = x \wedge 1$ and $g_2(x) = x/(1+x)$ for $x\geq 0$, both of which satisfy assumption (C1). The first layer of the proposed MIRG is a multigraph that limits the extent to which large weights encourage a higher number of links. The second layer of the MIRG is a version of the generalized random graph that behaves similarly to the first layer, except in a more continuous fashion. In order to estimate $\alpha$ from the degrees of the MIRG model, we naturally employ the inverse of the Hill estimator
\begin{align*}
H_{k_n, n} = \frac{1}{k_n} \sum_{i = 1}^n \log\left(R_{(i)}(n)/R_{(k_n + 1)}(n) \right),
\end{align*}
where we set $R_i$ to be the $\ell_1$ norm of the degree vector $(D_{i1}(n), D_{i2}(n))$, or the total degree across the layers. For each $\alpha \in \{1, 1.2, 1.4, 1.6, 1.8, 2\}$, we generate $1{,}000$ networks and for each simulated network, we compute $H^{-1}_{k_n, n}$ across $k_n \in \{100, 200, 500, 10^3, 5 \times 10^3, 10^4, 10^5 \}$. Table \ref{tab:hillsims} presents the bias and MSE (in parentheses) for each $\alpha$ and $k_n$ combination across the $1{,}000$ networks. Bold values indicate that the choice of $k_n$ minimizes the empirical absolute bias or MSE for the given value of $\alpha$. Note that our theory does not cover the case $\alpha = 1$.  Table \ref{tab:hillsims} indicates that consistency of the Hill estimator actually may not hold for $\alpha = 1$ since the empirical biases are all over $0.1$, regardless of the value of $k_n$. Note that as $\alpha$ increases, however, the value of $k_n$ that minimizes the empirical asbolute bias tends to decrease. This suggests than an assumption that controls the rate at which $k_n$ diminishes as $\alpha$ increases may indeed be necessary. On the other hand, the value of $k_n$ that minimizes the empirical MSE tends to stay relatively stable at $k_n = 5{,}000$.

\begin{table}
\centering
\begin{tabular}{| c | c c c c c c|}
\hline
$k_n$ & \multicolumn{6}{c|}{$\alpha$} \\
 & $1.0$ & $1.2$ & $1.4$ & $1.6$ & $1.8$ & $2.0$ \\
\hline
$100$ & 0.157 (0.039) & 0.077 (0.021) & 0.036 (0.022) & 0.021 (0.028) & 0.029 (0.039) & 0.028 (0.043) \\
$200$ & 0.141 (0.027) & 0.062 (0.011) & 0.026 (0.011) & 0.014 (0.013) & \textbf{0.021} (0.019) & \textbf{0.026} (0.020) \\
$500$ & 0.127 (0.018) & 0.051 (0.005) & 0.020 (0.004)  & \textbf{0.013} (0.005) & 0.022 (0.008) & 0.034 (0.010)  \\
$10^3$ & 0.117 (0.015) & 0.044 (0.003) & \textbf{0.018} (0.002) & 0.015 (0.003) & 0.027 (\textbf{0.005}) & 0.045 (0.007) \\
$5 \times 10^3$ & 0.103 (\textbf{0.012}) & 0.036 (\textbf{0.002}) & 0.020 (\textbf{0.001})  & 0.023 (\textbf{0.002}) & 0.078 (0.008) & 0.054 (\textbf{0.004}) \\
$10^4$ & \textbf{0.102} (0.014) & 0.036 (0.002) & 0.026 (0.002) & 0.072 (0.008) & 0.131 (0.018) & 0.134 (0.019) \\
$10^5$ & 0.117 (0.065) & \textbf{0.002} (0.007) & 0.164 (0.027) & 0.184 (0.035) & -0.140 (0.020) & 0.242 (0.059)  \\
\hline
\end{tabular}
\caption{Bias and MSEs (in parentheses) for $H^{-1}_{k_n, n}$ over $1{,}000$ iterations of a MIRG model with $n = 1{,}000{,}000$ and $\alpha \in \{1, 1.2, 1.4, 1.6, 1.8, 2 \}$. For each network, $H^{-1}_{k_n, n}$ is computed for $k_n \in \{100, 200, 500, 10^3, 5 \times 10^3, 10^4, 10^5 \}$. }
\label{tab:hillsims}
\end{table}

\section{Conclusion}\label{sec:conc}

In this paper we have introduced a multilayer inhomogeneous random graph (MIRG) model and studied its theoretical properties. The model is inherently flexible, allowing for layer-wise heterogeneity of edge counts and node attractiveness as well as complete specification of the asymptotic extremal dependence structure. When the latent weights are multivariate regularly varying, we also confirm that tail index estimation is feasible in such a model; with a restriction on the number of order statistics used in estimation, the Hill estimator based on the norm of the layer-wise degrees is consistent. Simulations indicate that, for practical sample sizes commonly encountered in network science, the asymptotic degree distribution accurately describes the behavior of extreme degrees in simulated networks. Additionally, the simulations indicate that a restriction on the number of order statistics is indeed necessary to achieve consistency of the Hill estimator in the MIRG model.

The proposed model and results produce additional research questions to consider in future work. First, we may consider the asymptotic behavior of the degree distribution when $\alpha  \in (0, 1]$ in the MIRG model. Here, the ratio of the maximum weight and sum of the weights converge in distribution, which should produce interesting behavior. We also conjecture that the MIRG model can be imbued with more complex weight distributions, such as independent and non-identically distributed data, and still produce scale-free behavior. From a statistical perspective, methods developed in \cite{einmahl2022extreme} and \cite{einmahl2023extreme} suggest that the Hill estimator would still be consistent in this case. Additionally, given a realization from the MIRG model, one would like to make inference on the extremal dependence structure. This includes estimating the support of the angular measure and distinguishing between different types of dependence. Since the tail empirical measure of the degree radii may be approximated by the tail empirical measure of independent data, methods developed in \cite{wang20232rv+} may apply. Such methods rely of asymptotic normality of the Hill estimator in the MIRG model, an important result to evaluate in its own right.

\section{Proofs}\label{sec:proofs}

\subsection{Weight approximation lemmas}\label{sec:weightlem}

In this section we present lemmas that help us establish the approximation \eqref{eq:degapprox}. The formal lemma is given in Lemma \ref{lem:gconv}, though we first present some helpful building block lemmas that will be used throughout the paper. Fix $l \in [L]$. From condition (C1), we expect $g_l(W_{1l}W_{jl}/T_l(n)) \approx c_l W_{1l}W_{jl}/T_l(n)$ to hold for all $j \in [n]$ when $W_{1l}W_{(1)l}/T_l(n)$ is small. Lemma \ref{lem:Wsmall} establishes the rate at which $W_{1l}W_{(1)l}(n)/T_l(n)$ tends to zero under condition (C3). Naturally, this requires that $W_{(1)l}(n)/T_l(n)$ tends to zero as $n \rightarrow \infty$. It is well known that,
\begin{align}
\label{eq:maxsum}
W_{(1)l}(n)/T_l(n) \xrightarrow{p} 0, \qquad \text{as } n \rightarrow \infty
\end{align}
if and only if $\mathbb{E}[W_{1l}] < \infty$ \cite{o1980limit}. This necessitates the requirement that $\alpha > 1$ in condition (C2). A useful result that helps characterize the rate of convergence of \eqref{eq:maxsum} is given in Theorem 9.1 of \cite{downey2007ratio}, which states that if the distribution of $W_{1l}$ is regularly varying with tail index $\alpha > 1$, then
\begin{align}
\label{eq:downey}
\mathbb{E}\left[W_{(1)l}(n)/T_l(n) \right] \sim \frac{\mathbb{E}[W_{(1)l}(n)]}{n E[W_{1l}]}, \qquad \text{as } n \rightarrow \infty.
\end{align}
We now present Lemma \ref{lem:Wsmall}. 
\begin{lemma}
\label{lem:Wsmall}
Let $\mathbf{W}_{[n]}$ be such that (C2) holds. Suppose $k_n$ is a sequence satisfying $k_n \rightarrow \infty$ and $k_n/n \rightarrow 0$ as $n \rightarrow \infty$. Further assume that (C3) holds. Fix $\epsilon > 0$ and suppose $\beta \in (0, \alpha)$. Then for each $l \in [L]$
\begin{align*}
&\frac{n}{k_n} \mathbb{P}\left(W^\beta_{1l}\frac{W_{(1)l}(n)}{T_l(n)} > \epsilon \right) \rightarrow 0,
\end{align*}
as $n \rightarrow \infty$.
\end{lemma}

\begin{proof}
Note that
\begin{align*}
W^\beta_{1l}\frac{W_{(1)l}(n)}{T_l(n)} \leq& \frac{W^{1 + \beta}_{1l}}{T_l(n)} + \frac{W^\beta_{1l}\underset{i \in [n]\setminus \{1\}}{\max} W_{il}}{T_l(n)} \\
\leq& \frac{W^{1 + \beta}_{1l}}{\sum_{i \in [n]\setminus \{1\}} W_{il}} + \frac{W^\beta_{1l}\underset{i \in [n]\setminus \{1\}}{\max} W_{il}}{\sum_{i \in [n]\setminus \{1\}} W_{il}}.
\end{align*}
In addition it is easily seen that since the distribution of $W_{1l}$ is continuous
\begin{align}
\label{eq:max1}
\frac{n}{k_n}\mathbb{P}\left(W_{1l} > \underset{i \in [n]\setminus \{1\}}{\max} W_{il} \right) = \frac{n}{k_n} \frac{1}{n} = \frac{1}{k_n} \rightarrow 0, \qquad \text{as } n \rightarrow \infty.
\end{align}
Hence when $W_{1l} \leq \underset{i \in [n]\setminus \{1\}}{\max} W_{il}$, we may bound
\begin{align}
\label{eq:maxW}
W^\beta_{1l}\frac{W_{(1)l}(n)}{T_l(n)} \leq& \frac{2 W^\beta_{1l}\underset{i \in [n]\setminus \{1\}}{\max} W_{il}}{\sum_{i \in [n]\setminus \{1\}} W_{il}}.
\end{align}
Applying \eqref{eq:max1} and \eqref{eq:maxW} thus gives
\begin{align*}
\frac{n}{k_n}\mathbb{P}\left( W^\beta_{1l}\frac{W_{(1)l}(n)}{T_l(n)} > \epsilon \right) \leq& \frac{n}{k_n}\mathbb{P}\left(W_{1l} > \underset{i \in [n]\setminus \{1\}}{\max} W_{il} \right) + \frac{n}{k_n}\mathbb{P}\left(\frac{2W^\beta_{1l}\underset{i \in [n]\setminus \{1\}}{\max} W_{il}}{\sum_{i \in [n]\setminus \{1\}} W_{il}} > \epsilon \right).
\end{align*}
The first term on the right-hand side of the inequality converges to $0$ by \eqref{eq:max1}. The convergence of the second term to $0$ is given Lemma \ref{lem:Ox2}.
\end{proof}

\begin{lemma}
\label{lem:Ox2}
Let $\mathbf{W}_{[n]}$ be such that (C2) holds. Suppose $k_n$ is a sequence satisfying $k_n \rightarrow \infty$ and $k_n/n \rightarrow 0$ as $n \rightarrow \infty$. Further assume that (C3) holds. Fix $\epsilon > 0$ and suppose $\beta \in (0, \alpha)$. Then for each $l \in [L]$
\begin{align}
\label{eq:Ox2}
&\frac{n}{k_n} \mathbb{P}\left(W^\beta_{1l} \frac{\underset{i \in [n]\setminus \{1\}}{\max} W_{il}}{\sum_{i \in [n]\setminus \{1\}} W_{il}} > \epsilon\right) \rightarrow 0,
\end{align}
as $n \rightarrow \infty$.
\end{lemma}

\begin{proof}
An application of Markov's inequality and independence across the weights gives that
\begin{align*}
\frac{n}{k_n} \mathbb{P}\left(W^\beta_{1l} \frac{\underset{i \in [n]\setminus \{1\}}{\max} W_{il}}{\sum_{i \in [n]\setminus \{1\}} W_{il}} > \epsilon\right) \leq& \epsilon^{-1} \mathbb{E}[W^\beta_{1l}] \frac{n}{k_n}\mathbb{E}\left[  \frac{\underset{i \in [n]\setminus \{1\}}{\max} W_{il}}{\sum_{i \in [n]\setminus \{1\}} W_{il}}  \right] \\
\intertext{and since $W_{1l}$ is regularly varying, we may apply Theorem 9.1 of \cite{downey2007ratio} to achieve that} 
\sim& \epsilon^{-1} \mathbb{E}[W^\beta_{1l}] \frac{n}{k_n}  \frac{\mathbb{E}\left[\underset{i \in [n]\setminus \{1\}}{\max} W_{il} \right]}{(n-1)\mathbb{E}\left[ W_{1l} \right]}  \\
\intertext{and applying Proposition 2.1 of \cite{resnick2008extreme}} 
\sim& \epsilon^{-1} \frac{\mathbb{E}[W^\beta_{1l}]}{\mathbb{E}\left[ W_{1l} \right]} \frac{1}{k_n}\Gamma\left(1 -\frac{1}{\alpha}\right) b^\star(n),
\end{align*}
where $b^\star(t)$ is the $1-1/t$ quantile function of the distribution of $W_{1l}$ for $t \geq 1$.  Since $b^\star(t)$ is regularly varying with index $1/\alpha$, we may apply (C3) to obtain \eqref{eq:Ox2}.
\end{proof}

We now present the lemma that formalizes the approximation \eqref{eq:degapprox}. 

\begin{lemma}
\label{lem:gconv}
Assume (C1) and that $\mathbb{E}[\|\mathbf{W}_1\|_1] < \infty$. Then for every $l \in [L]$,
\begin{align}
\label{eq:convp}
\sum_{j = 1}^n g_l(W_{1l}W_{jl}/T_l(n)) \xrightarrow{p} c_lW_{1l}, \qquad \text{as } n \rightarrow \infty.
\end{align}
\end{lemma}

\begin{proof}
Let $l \in [L]$. Recall that since $\mathbb{E}[\|\mathbf{W}_1\|_1] < \infty$, 
\begin{align*}
W_{(1)l}(n)/T_l(n) \xrightarrow{p} 0, \qquad \text{as } n \rightarrow \infty.
\end{align*}
In particular, since $W_{1l}$ is finite almost surely.
\begin{align}
\label{eq:conv1}
W_{1l}W_{(1)l}(n)/T_l(n) \xrightarrow{p} 0.
\end{align}
Recall $M$ and $\delta$ from \eqref{eq:glin}. From condition (C1), when $W_{1l}W_{(1)l}(n)/T_l(n) < \delta$
\begin{align*}
\left|\sum_{j = 1}^n g_l(W_{1l}W_{jl}/T_l(n)) - c_lW_{1l}\right| \leq& M W^{1 + \nu}_{1l} \sum_{j = 1}^n  W^{1 + \nu}_{jl}/T^{1 + \nu}_l(n) \\
\leq& M W^{1 + \nu}_{1l}  W^{\nu}_{(1)l}(n)/T^{\nu}_l(n).
\end{align*}
Fix $\epsilon > 0$. Then
\begin{align*}
\mathbb{P}\left(  \left|\sum_{j = 1}^n g_l(W_{1l}W_{jl}/T_l(n)) - c_lW_{1l}\right| > \epsilon \right) \leq &\mathbb{P}\left( M W^{1 + \nu}_{1l}  W^{\nu}_{(1)l}(n)/T^{\nu}_l(n) > \epsilon \right) \\
&+  \mathbb{P}\left( W_{1l}W_{(1)l}(n)/T_l(n) \geq \delta \right).
\end{align*}
Again, since $W^{1 + \nu}_{11}$ is finite almost surely, as $n \rightarrow \infty$
\begin{align}
\label{eq:conv2}
W^{1 + \nu}_{1l}W^{\nu}_{(1)l}(n)/T^{\nu}_l(n) \xrightarrow{p} 0.
\end{align}
Hence applying \eqref{eq:conv1} and \eqref{eq:conv2} completes the proof.
\end{proof}

\subsection{Coupling strategy}\label{sec:couple}

In this section we present a coupling that is used to translate results regarding the degrees in the more theoretically amenable Poisson $\mathcal{L}_1$ layers to the degrees in the Bernoulli $\mathcal{L}_2$ layers in the MIRG. Let $\mathcal{B}_{X \mid Y}$ denote the law of a random variable $X$ conditioned on the random variable $Y$. Construct a coupled version of $\mathbf{A}(n)$, $\tilde{\mathbf{A}}(n) = \{ \tilde{A}_{ijl} \}_{ijl}$, as follows. For $l \in \mathcal{L}_1$, set $\tilde{A}_{ijl} = A_{ijl}$ for $1 \leq i \leq j \leq n$. For $l \in \mathcal{L}_2$, let 
\begin{align*}
\tilde{A}_{ijl} \mid \mathbf{W}_{[n]} \sim \text{Poisson}(g_l(W_{il}W_{jl}/T_l(n))) \qquad \text{independently, }  1 \leq i \leq j \leq n,
\end{align*}
where
\begin{align}
\label{eq:dtv}
d_{\text{TV}}\left( \mathcal{B}_{ \tilde{A}_{ijl} \mid \mathbf{W}_{[n]}}, \mathcal{B}_{ A_{ijl} \mid \mathbf{W}_{[n]}} \right) \leq&  g_l^2\left(\frac{W_{il}W_{jl}}{T_{l}(n)}\right), \qquad 1 \leq i \leq j \leq n.
\end{align}
Such a coupling is given by a maximal coupling between Poisson and Bernoulli random variables \citep[see Section 2.2 of][]{van2017random}. Conditional on $\mathbf{W}_{[n]}$, we enforce independence of the pairs $(\tilde{A}_{ijl}, A_{ijl})$ across edges and layers. Under said coupling, one may easily compute
\begin{align}
\label{eq:expcouple}
\mathbb{E}\left[\left|\tilde{A}_{ijl} - A_{ijl} \right| \ \bigg\vert \ \mathbf{W}_{[n]} \right]\leq& K g_l^2\left(\frac{W_{il}W_{jl}}{T_{l}(n)}\right) \qquad 1 \leq i \leq j \leq n,
\end{align}
for some $K > 0$. Additionally, Theorem 2.10 of \cite{van2017random} gives that
\begin{align}
\label{eq:dtvsum}
d_{\text{TV}}\left( \mathcal{B}_{ \sum_{j = 1}^n \tilde{A}_{ijl} \mid \mathbf{W}_{[n]}}, \mathcal{B}_{ \sum_{j = 1}^n A_{ijl} \mid \mathbf{W}_{[n]}} \right) \leq&  \sum_{j = 1}^n g_l^2\left(\frac{W_{il}W_{jl}}{T_{l}(n)}\right), \qquad i \in [n].
\end{align}
 Let the degree of node $i$ in the coupled graph be given by $\tilde{\mathbf{D}}_i = (\tilde{D}_{i1}, \dots, \tilde{D}_{iL})$ where $\tilde{D}_{il} = \sum_{j = 1}^n \tilde{A}_{ijl}$ for $l \in [L]$. Our first lemma of this section provides a useful bound for the right-hand side of \eqref{eq:dtvsum} that we will repeatedly refer to. 
 
\begin{lemma}
\label{lem:sumgsqr}
Suppose (C1) holds. There exists a constant $C > 0$ such that for any $i \in [n]$ and $l \in [L]$ with $W_{il}\frac{W_{(1)l}(n)}{T_l(n)} < \delta$
\begin{align*}
\sum_{j = 1}^n g_l^2\left(\frac{W_{il}W_{jl}}{T_{l}(n)}\right) \leq C W^{2}_{il} \frac{W_{(1)l}(n)}{T_{l}(n)}.
\end{align*}
\begin{proof}
From (C1) and \eqref{eq:glin}, we have that if $x < \delta$, then
\begin{align*}
g^2_l(x) =& (g_l(x) - c_lx)^2 + c^2_lx^2 + 2c_lx(g_l(x) - c_lx) \\
\leq& M^2 x^{2 + 2\nu} + c^2_lx^2 + 2c_lMx^{2 + \nu}.
\end{align*}
Hence
\begin{align*}
\sum_{j = 1}^n&  g_l^2\left(\frac{W_{il}W_{jl}}{T_{l}(n)}\right) \\
\leq&  \sum_{j = 1}^n \left(M^2 \left(\frac{W_{il}W_{jl}}{T_{l}(n)}\right)^{2 + 2\nu} +  c_l^2 \left(\frac{W_{il}W_{jl}}{T_{l}(n)}\right)^2 + 2 c_l M \left(\frac{W_{il}W_{jl}}{T_{l}(n)}\right)^{2 + \nu} \right) \\
\leq&  M^2W^{2 + 2\nu}_{il} \left(\frac{W_{(1)l}(n)}{T_{l}(n)}\right)^{1 + 2\nu} + c_l^2 W^2_{il} \frac{W_{(1)l}(n)}{T_{l}(n)} + 2 c_l M W^{2 + \nu}_{il} \left(\frac{W_{(1)l}(n)}{T_{l}(n)}\right)^{1 + \nu} \\
\leq& M^2 \delta^{2\nu} W^{2}_{il} \frac{W_{(1)l}(n)}{T_{l}(n)} + c_l^2 W^2_{il} \frac{W_{(1)l}(n)}{T_{l}(n)} + 2 c_l M \delta^\nu W^{2}_{il} \frac{W_{(1)l}(n)}{T_{l}(n)} \\
\leq& C  W^{2}_{il} \frac{W_{(1)l}(n)}{T_{l}(n)},
\end{align*}
where we have chosen $C = M^2 \delta^{2\nu} + \max_{l \in [L]} c_l^2 + 2 M \delta^\nu \max_{l \in [L]} c_l$.
\end{proof}
\end{lemma}

\subsection{Properties of the degree distribution}

In this section we present proofs of Lemma \ref{lem:typdeg} and Theorem \ref{thm:rv}. The proofs rely on tools developed in Sections \ref{sec:weightlem} and \ref{sec:couple}

\subsubsection{Proof of Lemma \ref{lem:typdeg}}\label{sec:typdegproof}

\begin{proof}
See that by conditioning on $\{ \mathbf{W}_i \}_{i = 1}^n$
\begin{align*}
\mathbb{P}\left(\mathbf{D}_1(n) = \mathbf{m} \right) =&\mathbb{E}\left[ \mathbb{P}\left(\mathbf{D}_1(n) = \mathbf{m} \mid \mathbf{W}_{[n]} \right) \right] \\
=& \mathbb{E}\left[ \prod_{l = 1}^L \mathbb{P}\left(D_{1l}(n) = m_l \mid \mathbf{W}_{[n]} \right) \right].
\end{align*}
For $l \in \mathcal{L}_1$, see that by applying Lemma \ref{lem:gconv}
\begin{align*}
\mathbb{P}\left(D_{1l}(n) = m_l \mid \mathbf{W}_{[n]} \right) =& \frac{\left(\sum_{j = 1}^n g_l(W_{1l}W_{jl}/T_l(n))\right)^{m_l}e^{-\sum_{j = 1}^n g_l(W_{1l}W_{jl}/T_l(n))}}{m_l!} \\
\xrightarrow{p}& \frac{\left(c_l W_{1l}\right)^{m_l}e^{-c_l W_{1l}}}{m_l!}, \qquad \text{as } n \rightarrow \infty.
\end{align*}
For $l \in \mathcal{L}_2$, we employ the coupled version of $\mathbf{A}(n)$, $\tilde{\mathbf{A}}(n)$, developed in Section \ref{sec:couple}.  
From \eqref{eq:dtvsum} and Lemma \ref{lem:sumgsqr}, we have that
\begin{align*}
d_{\text{TV}}\left( \mathcal{B}_{ \sum_{j = 1}^n \tilde{A}_{1jl} \mid \mathbf{W}_{[n]}}, \mathcal{B}_{ \sum_{j = 1}^n A_{1jl} \mid \mathbf{W}_{[n]}} \right) \leq&  \sum_{j = 1}^n g_l^2\left(\frac{W_{1l}W_{jl}}{T_{l}(n)}\right),
\end{align*}
and when $W_{1l} \frac{W_{(1)l}(n)}{T_{l}(n)} < \delta$, by Lemma \ref{lem:sumgsqr}
\begin{align*}
d_{\text{TV}}\left( \mathcal{B}_{ \sum_{j = 1}^n \tilde{A}_{1jl} \mid \mathbf{W}_{[n]}}, \mathcal{B}_{ \sum_{j = 1}^n A_{1jl} \mid \mathbf{W}_{[n]}} \right) \leq C W^{2}_{1l} \frac{W_{(1)l}(n)}{T_{l}(n)}.
\end{align*}
Hence, for any $\epsilon > 0$
\begin{align*}
\mathbb{P}\left(d_{\text{TV}}\left( \mathcal{B}_{ \sum_{j = 1}^n \tilde{A}_{1jl} \mid \mathbf{W}_{[n]}}, \mathcal{B}_{ \sum_{j = 1}^n A_{1jl} \mid \mathbf{W}_{[n]}} \right) > \epsilon \right) \leq &\mathbb{P}\left(C W^{2}_{1l} \frac{W_{(1)l}(n)}{T_{l}(n)} > \epsilon \right) \\
&+ \mathbb{P}\left(W_{1l} \frac{W_{(1)l}(n)}{T_{l}(n)} \geq \delta \right).
\end{align*}
Recalling that $\mathbb{E}[\|\mathbf{W}_1\|_1] < \infty$, $W^j_{1l}W_{(1)l}(n)/T_{l}(n)$  tends to $0$ in probability since $W_{(1)l}(n)/T_{l}(n)$ tends to zero in probability and $W^j_{1l}$ is finite almost surely for $j = 1, 2$. Hence for $l \in \mathcal{L}_2$,
\begin{align*}
\mathbb{P}\left(D_{1l}(n) = m_l \mid \mathbf{W}_{[n]} \right)  \xrightarrow{p} \frac{\left(c_l W_{1l}\right)^{m_l}e^{-c_l W_{1l}}}{m_l!},
\end{align*}
as $n \rightarrow \infty$. Thus, dominated convergence gives that as $n \rightarrow \infty$
\begin{align*}
\mathbb{P}\left(\mathbf{D}_1(n) = \mathbf{m} \right) =& \mathbb{E}\left[ \prod_{l = 1}^L \mathbb{P}\left(D_{1l}(n) = m_l \mid \mathbf{W}_{[n]} \right) \right] \rightarrow \mathbb{E}\left[ \prod_{l = 1}^L \frac{\left(c_l W_{1l}\right)^{m_l}e^{-c_l W_{1l}}}{m_l!} \right].
\end{align*}
\end{proof}

\subsubsection{Proof of Theorem \ref{thm:rv}}\label{sec:rvproof}

In this section we prove Theorem \ref{thm:rv}. Given that $\mathbb{P}\left(\boldsymbol{\mathcal{D}}_1 \in \cdot \right)$ in Lemma \ref{lem:typdeg} emits a mixed-Poisson representation, the proof of of Theorem \ref{thm:rv} relies heavily on the concentration of the Poisson distribution for large rates. Such properties are also used in the single-layer case \citep[see (A4) and Section 4 of][]{bhattacharjee2022large}. We now recall concentration results for the Poisson distribution. 

\begin{lemma}
\label{lem:mbound}
Suppose $X \sim \text{Poisson}(\lambda)$. For each $m \in \mathbb{N}$, there exists nonnegative constants $a_{m}$ and $C_m$ depending only on $m$ such that
\begin{align*}
\mathbb{E}\left( \left|X - \lambda \right|^m\right) \leq a_m \lambda^{m/2} + C_m. 
\end{align*}
\end{lemma}
In order to form a multivariate extension of Lemma \ref{lem:mbound}, we note that for $X_l \sim \text{Poisson}(\lambda_l)$ for $l = 1, \dots, L$, $\mathbf{X} = (X_1, \dots, X_L)$ and $\boldsymbol{\lambda} = (\lambda_1, \dots, \lambda_L)$
\begin{align}
\label{lem:mnormequal}
\mathbb{E}\left( \|\mathbf{X} - \boldsymbol{\lambda} \|^m_m \right) \leq a_m \| \boldsymbol{\lambda} \|^{m/2}_{m/2} + LC_m,
\end{align}
which, when combined with the equivalence $\ell_p$ norms, gives the following lemma.  
\begin{lemma}
\label{lem:mnorm}
Suppose $X_l \sim \text{Poisson}(\lambda_l)$ for $l = 1, \dots, L$. Let $\mathbf{X} = (X_1, \dots, X_L)$ and $\boldsymbol{\lambda} = (\lambda_1, \dots, \lambda_L)$. For each $m, p \in \mathbb{N}$ there exists nonnegative constants $a_{m}$ and $C_{m}$ depending only on $m$ such that
\begin{align*}
\mathbb{E}\left( \|\mathbf{X} - \boldsymbol{\lambda} \|_p^m  \right) \leq L^{\frac{m}{p} + 1}\left(  a_m \| \boldsymbol{\lambda} \|^{m/2}_{p} + C_m \right).
\end{align*}
\end{lemma}

Which such results in hand, we are now prepared to prove Theorem \ref{thm:rv}. 

\begin{proof}[Proof of Theorem \ref{thm:rv}]
We only prove part (b) since  part (a) is simpler and proven similarly. Let $f \in \mathcal{C}(\mathbb{R}_+^L\setminus \mathbb{C}_0)$ be uniformly continuous. It suffices to show that
\begin{align}
\label{eq:mapprox}
t\left| \mathbb{E}[f(\boldsymbol{\mathcal{D}}_1/b_0(t))] - \mathbb{E}[f(\boldsymbol{\mathcal{W}}_1/b_0(t))] \right| \rightarrow 0, \qquad \text{as } t \rightarrow \infty.
\end{align}
Since $f \in \mathcal{C}(\mathbb{R}_+^L\setminus \mathbb{C}_0)$, it has support bounded away from $\mathbb{C}_0$. That is, there exists an $\theta > 0$ such that $f(\mathbf{x}) = 0$ for all $\mathbf{x} \in \mathbb{R}_+^L\setminus \mathbb{C}_0$ with $d_p(\mathbf{x}, \mathbb{C}_0) \leq \theta$. Thus
\begin{align*}
t&\left| \mathbb{E}[f(\boldsymbol{\mathcal{D}}_1/b_0(t))] - \mathbb{E}[f(\boldsymbol{\mathcal{W}}_1/b_0(t))] \right| \\
\leq& t\mathbb{E}\left[  \left| f(\boldsymbol{\mathcal{D}}_1/b_0(t)) - f(\boldsymbol{\mathcal{W}}_1/b_0(t)) \right|\right] \\
=& t\mathbb{E}\left[ \left| f(\boldsymbol{\mathcal{D}}_1/b_0(t)) - f(\boldsymbol{\mathcal{W}}_1/b_0(t))\right|  1_{\left\lbrace d_p(\boldsymbol{\mathcal{D}}_1/b_0(t), \mathbb{C}_0) > \theta,  d_p(\boldsymbol{\mathcal{W}}_1/b_0(t), \mathbb{C}_0) > \theta  \right\rbrace}  \right] \\.
&+ t\mathbb{E}\left[f(\boldsymbol{\mathcal{D}}_1/b_0(t))1_{\left\lbrace d_p(\boldsymbol{\mathcal{D}}_1/b_0(t), \mathbb{C}_0) > \theta,  d_p(\boldsymbol{\mathcal{W}}_1/b_0(t), \mathbb{C}_0) \leq \theta  \right\rbrace}  \right] \\
&+ t\mathbb{E}\left[f(\boldsymbol{\mathcal{W}}_1/b_0(t))1_{\left\lbrace d_p(\boldsymbol{\mathcal{D}}_1/b_0(t), \mathbb{C}_0) \leq \theta,  d_p(\boldsymbol{\mathcal{W}}_1/b_0(t), \mathbb{C}_0) > \theta  \right\rbrace}  \right] \\
\equiv& A_1(t) + A_2(t) + A_3(t).
\end{align*}
We show that $A_1(t), A_2(t)$ and $A_3(t)$ all tend to zero as $t \rightarrow \infty$. For $A_1(t)$, recall that, conditional on $\boldsymbol{\mathcal{W}}_1$, the elements of $\boldsymbol{\mathcal{D}}_1$ are conditionally independent Poisson random variables. Hence we may apply Lemma \ref{lem:mnorm} to supply the existence of constants $a_{p}$ and $C_{p}$ such that, almost surely
\begin{align}
\label{eq:lemapp}
\mathbb{E}\left[ \|\boldsymbol{\mathcal{D}}_1 - \boldsymbol{\mathcal{W}}_1 \|_p^p \mid \boldsymbol{\mathcal{W}}_1 \right] \leq L^{2}\left(  a_p \| \boldsymbol{\mathcal{W}}_1 \|^{p/2}_{p} + C_p \right).
\end{align}
Fix $\epsilon > 0$. Define the random variable 
\begin{align*}
\delta_p = \delta_p(\boldsymbol{\mathcal{W}}_1, t) = \frac{L^{2/p}(a_{p} \| \boldsymbol{\mathcal{W}}_1 \|^{p/2}_{p} + C_{p})^{1/p}}{b_0(t)\epsilon^{1/p}},
\end{align*}
Note that if $\boldsymbol{\mathcal{D}}_1/b_0(t), \boldsymbol{\mathcal{W}}_1/b_0(t)$ are bounded away from $\mathbb{C}_0$ and
\begin{align*}
\|\boldsymbol{\mathcal{D}}_1/b_0(t) - \boldsymbol{\mathcal{W}}_1/b_0(t) \|_p \leq \delta_p,
\end{align*}
then
\begin{align*}
\left| f(\boldsymbol{\mathcal{D}}_1/b_0(t)) - f(\boldsymbol{\mathcal{W}}_1/b_0(t))\right| \leq \Delta_f^p(\delta_p).
\end{align*}
Otherwise, since $f \in \mathcal{C}(\mathbb{R}_+^L\setminus \mathbb{C}_0)$, there exists an $B > 0$ such that for all $\mathbf{x} \in \mathbb{R}_+^L\setminus \mathbb{C}_0$, $f(\mathbf{x}) \leq B$. Thus we may bound $A_1(t)$ by
\begin{align*}
t\mathbb{E}&\left[ \left| f(\boldsymbol{\mathcal{D}}_1/b_0(t)) - f(\boldsymbol{\mathcal{W}}_1/b_0(t))\right|  1_{\left\lbrace  d_p(\boldsymbol{\mathcal{D}}_1/b_0(t), \mathbb{C}_0) > \theta,  d_p(\boldsymbol{\mathcal{W}}_1/b_0(t), \mathbb{C}_0) > \theta  \right\rbrace}  \right] \\
\leq& t\mathbb{E}\left[\Delta_f^p(\delta_p) 1_{\left\lbrace d_p(\boldsymbol{\mathcal{W}}_1/b_0(t), \mathbb{C}_0) > \theta  \right\rbrace}  \right] + 2B t\mathbb{E}\left[ 1_{\left\lbrace\|\boldsymbol{\mathcal{D}}_1/b_0(t) - \boldsymbol{\mathcal{W}}_1/b_0(t) \|_p > \delta_p \right\rbrace}1_{\left\lbrace d_p(\boldsymbol{\mathcal{W}}_1/b_0(t), \mathbb{C}_0) > \theta  \right\rbrace}\right] \\
=& t\mathbb{E}\left[\Delta_f^p(\delta_p) 1_{\left\lbrace d_p(\boldsymbol{\mathcal{W}}_1, \mathbb{C}_0) > \theta b_0(t)  \right\rbrace}  \right] + 2B t\mathbb{E}\left[ 1_{\left\lbrace\|\boldsymbol{\mathcal{D}}_1/b_0(t) - \boldsymbol{\mathcal{W}}_1/b_0(t) \|_p > \delta_p \right\rbrace}1_{\left\lbrace d_p(\boldsymbol{\mathcal{W}}_1, \mathbb{C}_0) > \theta  b_0(t) \right\rbrace}\right] \\
\equiv& A_{11}(t) + A_{12}(t),
\end{align*}
where we have noted that
\begin{align*}
\left\lbrace \mathbf{x} \in \mathbb{R}_+^L : \inf_{\mathbf{y} \in \mathbb{C}_0} \left\| \frac{\mathbf{x}}{b_0(t)} - \mathbf{y} \right\|_p > \eta \right\rbrace =& \left\lbrace \mathbf{x} \in \mathbb{R}_+^L : \inf_{\mathbf{y} \in \mathbb{C}_0} \left\| \mathbf{x} - b_0(t)\mathbf{y} \right\|_p > \eta b_0(t) \right\rbrace \\
=& \left\lbrace \mathbf{x} \in \mathbb{R}_+^L : \inf_{\mathbf{z} \in \mathbb{C}_0} \left\| \mathbf{x} - \mathbf{z} \right\|_p > \eta b_0(t) \right\rbrace,
\end{align*}
since $\mathbb{C}_0$ is a cone. Hence in order to prove that $A_1(t)$ tends to $0$ as $t \rightarrow \infty$, it suffices to show that $A_{11}(t)$ and $A_{12}(t)$ converge to $0$ as $t \rightarrow \infty$. For $A_{11}(t)$, see that since $\Delta_f^p(\delta) \leq 2B$ for any $\delta > 0$ we have that $A_{11}(t)$ is equivalent to
\begin{align*}
t\mathbb{E}&\left[\Delta_f^p(\delta_p) 1_{\left\lbrace d_p(\boldsymbol{\mathcal{W}}_1, \mathbb{C}_0) > \theta b_0(t)  \right\rbrace}  \right] \\
=& t\mathbb{E}\left[\Delta_f^p(\delta_p) 1_{\left\lbrace d_p(\boldsymbol{\mathcal{W}}_1, \mathbb{C}_0) > \theta b_0(t) \right\rbrace}1_{\left\lbrace \delta_p \leq \epsilon\right\rbrace}  \right] + t\mathbb{E}\left[\Delta_f^p(\delta_p) 1_{\left\lbrace d_p(\boldsymbol{\mathcal{W}}_1, \mathbb{C}_0) > \theta b_0(t)  \right\rbrace}1_{\left\lbrace \delta_p > \epsilon\right\rbrace}  \right] \\
\leq& \Delta_f^p(\epsilon) t\mathbb{P}\left( d_p(\boldsymbol{\mathcal{W}}_1, \mathbb{C}_0) > \theta b_0(t)  \right) + 2B t\mathbb{P}\left(\delta_p > \epsilon  \right) \\
=& \Delta_f^p(\epsilon) t\mathbb{P}\left( d_p(\boldsymbol{\mathcal{W}}_1, \mathbb{C}_0) > \theta b_0(t)  \right) + 2B t\mathbb{P}\left(L^{2/p}\left(a_{p}\|\boldsymbol{\mathcal{W}}_1\|_{p}^{p/2} + C_{p} \right)^{1/p} > b_0(t)\epsilon^{(p+1)/p}  \right),
\end{align*}
Due to regular variation of $\boldsymbol{\mathcal{W}}_1$ on $\mathbb{R}_+^L \setminus \mathbb{C}_0$, we have that by the GPOLAR transform as $t \rightarrow \infty$
\begin{align}
\label{eq:A1_1}
 t\mathbb{P}\left( d_p(\boldsymbol{\mathcal{W}}_1, \mathbb{C}_0) > \theta b_0(t)  \right) \rightarrow \theta^{-\alpha_0}.
\end{align}
Additionally, since $\boldsymbol{\mathcal{W}}_1$ is multivariate regularly varying on $\mathbb{R}_+^L \setminus \{\mathbf{0} \}$ with tail index $\alpha$, $\| \boldsymbol{\mathcal{W}}_1 \|_{p}$ is univariate regularly varying with tail index $\alpha$. Thus the quantity $(a_{p} \| \boldsymbol{\mathcal{W}}_1\|^{p/2}_{p} + C_{p})^{1/p}$ is regularly varying with tail index $2\alpha$ and 
\begin{align}
\label{eq:A1_2}
t\mathbb{P}\left(L^{2/p}\left(a_{p}\|\boldsymbol{\mathcal{W}}_1 \|_{p}^{p/2} + C_{p} \right)^{1/p} > b_0(t)\epsilon^{(p+1)/p}  \right) \rightarrow 0, \qquad \text{as } t \rightarrow \infty,
\end{align}
since $b_0(t) \in RV_{1/\alpha_0}$ and $\alpha_0 \in [\alpha, 2\alpha)$. Thus, from \eqref{eq:A1_1} and \eqref{eq:A1_2}, we have that
\begin{align*}
\limsup_{t \rightarrow \infty} A_1(t) \leq \Delta_f^p(\epsilon) \theta^{-\alpha_0}.
\end{align*}
For $A_{12}(t)$, see that by tower property and Markov's inequality
\begin{align*}
A_{12}(t) =& 2Bt\mathbb{E}\left[ \mathbb{P}\left( \|\boldsymbol{\mathcal{D}}_1/b_0(t) - \boldsymbol{\mathcal{W}}_1/b_0(t) \|_p > \delta_p \mid \boldsymbol{\mathcal{W}}_1\right)1_{\left\lbrace d_p(\boldsymbol{\mathcal{W}}_1, \mathbb{C}_0) > \theta b_0(t) \right\rbrace }\right] \\
\leq& 2Bt \mathbb{E}\left[ \delta_p^{-p} \mathbb{E}\left[\|\boldsymbol{\mathcal{D}}_1/b_0(t) - \boldsymbol{\mathcal{W}}_1/b_0(t) \|^p_p \mid \boldsymbol{\mathcal{W}}_1 \right]1_{\left\lbrace d_p(\boldsymbol{\mathcal{W}}_1, \mathbb{C}_0) > \theta b_0(t) \right\rbrace } \right] \\
=& 2Bt \mathbb{E}\left[ (b_0(t)\delta_p )^{-p}\mathbb{E}\left[\|\boldsymbol{\mathcal{D}}_1 - \boldsymbol{\mathcal{W}}_1 \|^p_p \mid\boldsymbol{\mathcal{W}}_1 \right]1_{\left\lbrace d_p(\boldsymbol{\mathcal{W}}_1, \mathbb{C}_0) > \theta b_0(t)\right\rbrace } \right] \\
\leq& 2Bt \mathbb{E}\left[(b_0(t)\delta_p )^{-p}  L^2\left(a_{p}  \|\boldsymbol{\mathcal{W}}_1 \|_{p}^{p/2} + C_{p} \right)1_{\left\lbrace d_p(\boldsymbol{\mathcal{W}}_1, \mathbb{C}_0) > \theta b_0(t) \right\rbrace } \right] \\
=& 2B \epsilon t \mathbb{P}\left( d_p(\boldsymbol{\mathcal{W}}_1, \mathbb{C}_0) > \theta b_0(t) \right),
\end{align*}
where we have applied \eqref{eq:lemapp} in the second to last line and the definition of $\delta_p$ in the penultimate step. From \eqref{eq:A1_1}, we thus have that
\begin{align*}
\limsup_{t \rightarrow \infty} A_{12}(t) \leq 2B\epsilon  \theta^{-\alpha_0}.
\end{align*}
Thus, returning to our original goal, we achieve that 
\begin{align*}
\limsup_{t \rightarrow \infty} A_1(t) \leq& \limsup_{t \rightarrow \infty} A_{11}(t) + \limsup_{t \rightarrow \infty} A_{12}(t) \\
\leq& (\Delta_f^p(\epsilon) + 2K\epsilon) \theta^{-\alpha_0}.
\end{align*}
Since $\epsilon > 0$ was arbitrary, we have that $A_1(t) \rightarrow 0$ as $t \rightarrow \infty$. Fix $\eta \in (0, \theta)$. For $A_2(t)$ and $A_3(t)$, note that
\begin{align*}
A_2(t) \leq& Bt\mathbb{P}\left(d_p(\boldsymbol{\mathcal{D}}_1/b_0(t), \mathbb{C}_0) > \theta,  d_p(\boldsymbol{\mathcal{W}}_1/b_0(t), \mathbb{C}_0) \leq \theta  \right) \\
\leq&  Bt\mathbb{P}\left(d_p(\boldsymbol{\mathcal{W}}_1/b_0(t), \mathbb{C}_0) \in (\theta - \eta ,\theta]  \right) +  Bt\mathbb{P}\left(d_p(\boldsymbol{\mathcal{D}}_1, \mathbb{C}_0) - d_p(\boldsymbol{\mathcal{W}}_1, \mathbb{C}_0) > \eta b_0(t)  \right) \\
\leq& Bt\mathbb{P}\left(d_p(\boldsymbol{\mathcal{W}}_1/b_0(t), \mathbb{C}_0) \in (\theta - \eta ,\theta]  \right) +  Bt\mathbb{P}\left(\| \boldsymbol{\mathcal{D}}_1 - \boldsymbol{\mathcal{W}}_1 \|_p > \eta b_0(t)  \right) \\
=& A_{21}(t) + \tilde{A}(t),
\end{align*}
and similarly
\begin{align*}
A_3(t) \leq& Bt\mathbb{P}\left(d_p(\boldsymbol{\mathcal{D}}_1/b_0(t), \mathbb{C}_0) \leq \theta,  d_p(\boldsymbol{\mathcal{W}}_1/b_0(t), \mathbb{C}_0) > \theta  \right) \\
\leq&  Bt\mathbb{P}\left(d_p(\boldsymbol{\mathcal{W}}_1/b_0(t), \mathbb{C}_0) \in (\theta,\theta + \eta]  \right) +  Bt\mathbb{P}\left(d_p(\boldsymbol{\mathcal{D}}_1, \mathbb{C}_0) - d_p(\boldsymbol{\mathcal{W}}_1, \mathbb{C}_0) > \eta b_0(t)  \right) \\
\leq& Bt\mathbb{P}\left(d_p(\boldsymbol{\mathcal{W}}_1/b_0(t), \mathbb{C}_0) \in (\theta ,\theta + \eta]  \right) +  Bt\mathbb{P}\left(\| \boldsymbol{\mathcal{D}}_1 - \boldsymbol{\mathcal{W}}_1 \|_p > \eta b_0(t)  \right) \\
=& A_{31}(t) + \tilde{A}(t).
\end{align*}
Analyzing $\tilde{A}(t)$, let $s \in (\alpha_0, 2\alpha)$ and $m = \ceil{s}$. Then by Markov and Jensen inequality,
\begin{align*}
\tilde{A}(t) \leq& \frac{Bt}{\eta^s b^s_0(t)}\mathbb{E}\left[\mathbb{E}\left[ \| \boldsymbol{\mathcal{D}}_1 - \boldsymbol{\mathcal{W}}_1 \|^s_p \mid \boldsymbol{\mathcal{W}}_1 \right] \right] \\
\leq& \frac{Bt}{\eta^s b^s_0(t)}\mathbb{E}\left[\left(\mathbb{E}\left[ \| \boldsymbol{\mathcal{D}}_1 - \boldsymbol{\mathcal{W}}_1 \|^m_p \mid \boldsymbol{\mathcal{W}}_1 \right]\right)^{\frac{s}{m}}\right] \\
\leq& \frac{Bt}{\eta^s b^s_0(t)}\mathbb{E}\left[\left(L^{\frac{m}{p} + 1}\left(a_m \| \boldsymbol{\mathcal{W}}_1\|_p^{m/2} + C_m \right)\right)^{\frac{s}{m}}\right].
\end{align*}
The random variable inside the expectation is reguarly varying with tail index $2\alpha/s > 1$ and hence the expectation is finite. In addition, $b^s_0(t)$ is regularly varying with index $s/\alpha_0 > 1$ and thus $\tilde{A}(t) \rightarrow 0$ as $t \rightarrow \infty$. For $A_{21}(t)$ and $A_{31}(t)$, the GPOLAR transform gives that as $t \rightarrow \infty$,
\begin{align*}
&A_{21}(t) \rightarrow (\theta - \eta)^{-\alpha_0} - \theta^{-\alpha_0}, \\
&A_{31}(t) \rightarrow \theta^{-\alpha_0} - (\theta + \eta)^{-\alpha_0}.
\end{align*}
Since $\eta > 0$ is arbitrary, $A_2(t), A_3(t) \rightarrow 0$ as $t \rightarrow \infty$. Thus the proof is complete. 
\end{proof}

\subsubsection{Proof of \eqref{eq:example_final} in Example \ref{ex:example}}\label{sec:example}

In order to prove \eqref{eq:example_final}, it suffices to consider the vague convergence of measures; see for instance \cite[Chapter~3.3.5]{resnick2007heavy}. We include the definition below. Let $\mathbb{E}$ be a locally compact topological space that has a countable base with Borel sigma algebra $\mathcal{E}$. Let $M_+(\mathbb{E})$ be the space of Radon measures on $\mathcal{E}$. Then a sequence of measures $\mu_n \in M_+(\mathbb{E})$ converges vaguely to $\mu \in M_+(\mathbb{E})$, written $\mu_n \rightarrow \mu$, if for all nonnegative, compactly supported continuous functions on $\mathcal{E}$
\begin{align*}
\int f d\mu_n \rightarrow \int f d\mu, \qquad \text{as } n \rightarrow \infty. 
\end{align*}
With vague convergence in hand, we may now prove \eqref{eq:example_final}.
\begin{proof}[Proof of \eqref{eq:example_final}]
We first show that
\begin{align}
\label{eq:triconv}
t\mathbb{P}\left( \left(\frac{|\mathcal{D}_{11} - \mathcal{D}_{12}|}{\sqrt{2W_{11}}}, \frac{\sqrt{2}\mathcal{D}_{11}/\sqrt{W_{11}}}{ |\mathcal{D}_{11} - \mathcal{D}_{12}|}, \frac{\sqrt{W_{11}}}{b_0(t)} \right) \in \cdot \right)  \xrightarrow{v} (S \times \nu_{2\alpha})(\cdot),
\end{align}
in $M_+([0, \infty]^2 \times (0, \infty])$ where $S$ is the degenerate probability measure given by
\begin{align*}
S(\cdot) = \mathbb{P}\left((|Z|, 1/|Z|) \in \cdot \right).
\end{align*}
Since $\sqrt{W_{11}}$ is marginally regularly varying, it suffices to show that for any $x >0$ and bounded, continuous function $h$ on $[0, \infty]^2$
\begin{align}
\label{eq:weakconvg}
t\mathbb{E}\left[1_{\left\lbrace \sqrt{W_{11}} > b_0(t)x \right\rbrace} h\left(\frac{|\mathcal{D}_{11} - \mathcal{D}_{12}|}{\sqrt{2W_{11}}}, \frac{\sqrt{2}\mathcal{D}_{11}/\sqrt{W_{11}}}{ |\mathcal{D}_{11} - \mathcal{D}_{12}|} \right) \right] \rightarrow x^{-2\alpha}E\left[h\left(|Z|, \frac{1}{|Z|}\right) \right],
\end{align}
as $t \rightarrow \infty$. First note that for independent Poisson random variables $X, Y$ with common rate $\lambda$, it is easily seen that as $\lambda \rightarrow \infty$
\begin{align*}
|X-Y|/\sqrt{2\lambda} \xrightarrow{d} |Z|, \qquad  X/\lambda \xrightarrow{p} 1,
\end{align*}
where $Z$ is a standard normal random variable. Hence the joint weak convergence of $(|X-Y|/\sqrt{2\lambda},X/\lambda)$ is obtained and the continuous mapping theorem gives that as $\lambda \rightarrow \infty$
\begin{align*}
\left(\frac{|X-Y|}{\sqrt{2\lambda}}, \frac{\sqrt{2}X/\sqrt{\lambda}}{|X-Y|}\right) \xrightarrow{d} \left(|Z|,1/|Z|\right),
\end{align*}
which, in particular, implies that as $\lambda \rightarrow \infty$
\begin{align}
\label{eq:weakconv}
\mathbb{E}\left[h\left(\frac{|X-Y|}{\sqrt{2\lambda}}, \frac{\sqrt{2}X/\sqrt{\lambda}}{|X-Y|} \right) \right] \rightarrow \mathbb{E}\left[h\left(|Z|,1/|Z|\right) \right].
\end{align}
Fix $\epsilon > 0$. From \eqref{eq:weakconv}, we may extract the existence of a $\lambda^\star > 0$ such that
\begin{align}
\label{eq:weakconvbd}
\sup_{\lambda > \lambda^\star} \left| \mathbb{E}\left[h\left(\frac{|X-Y|}{\sqrt{2\lambda}}, \frac{\sqrt{2}X/\sqrt{\lambda}}{|X-Y|} \right) - h\left(|Z|,\frac{1}{|Z|} \right) \right] \right| < \epsilon.
\end{align}
Towards proving \eqref{eq:weakconvg} then, see that we may write
\begin{align*}
&t\left| \mathbb{E}\left[ 1_{\left\lbrace \sqrt{W_{11}} > b_0(t)x \right\rbrace}h\left(\frac{|\mathcal{D}_{11} - \mathcal{D}_{12}|}{\sqrt{2W_{11}}}, \frac{\sqrt{2}\mathcal{D}_{11}/\sqrt{W_{11}}}{ |\mathcal{D}_{11} - \mathcal{D}_{12}|} \right)\right] - x^{-2\alpha}E\left[h\left(|Z|, \frac{1}{|Z|}\right) \right] \right| \\
&\leq  t\left| \mathbb{E}\left[1_{\left\lbrace \sqrt{W_{11}} > b_0(t)x \right\rbrace}\left\lbrace h\left(\frac{|\mathcal{D}_{11} - \mathcal{D}_{12}|}{\sqrt{2W_{11}}}, \frac{\sqrt{2}\mathcal{D}_{11}/\sqrt{W_{11}}}{ |\mathcal{D}_{11} - \mathcal{D}_{12}|} \right)- h\left(|Z|, \frac{1}{|Z|}\right) \right\rbrace \right] \right| \\
& \ \ \ + \left| tE\left[1_{\left\lbrace \sqrt{W_{11}} > b_0(t)x \right\rbrace}h\left(|Z|, \frac{1}{|Z|}\right) \right] - x^{-2\alpha}E\left[h\left(|Z|, \frac{1}{|Z|}\right) \right] \right| \\
&= C_1(t) + C_2(t).
\end{align*}
Using the tower property, see that
\begin{align*}
C_1(t) \leq&  t\mathbb{E}\left[1_{\left\lbrace \sqrt{W_{11}} > b_0(t)x \right\rbrace}\left| \mathbb{E}\left[ h\left(\frac{|\mathcal{D}_{11} - \mathcal{D}_{12}|}{\sqrt{2W_{11}}}, \frac{\sqrt{2}\mathcal{D}_{11}/\sqrt{W_{11}}}{ |\mathcal{D}_{11} - \mathcal{D}_{12}|} \right)- h\left(|Z|, \frac{1}{|Z|}\right) \Bigg\vert \ W_{11} \right] \right| \right] \\
\intertext{and choosing $t$ large enough so that $b^2_0(t)x^2 > \lambda^\star$, we have that} 
\leq& \epsilon t\mathbb{P}\left(\sqrt{W_{11}} > b_0(t)x \right) \rightarrow \epsilon x^{-2\alpha},
\end{align*}
as $t \rightarrow \infty$. Hence, $C_1(t) \rightarrow 0$ as $t \rightarrow \infty$. Since $h$ is bounded, we may apply regular variation of $\sqrt{W_{11}}$ to conclude that $C_2(t) \rightarrow 0$ as $t \rightarrow \infty$. Hence \eqref{eq:weakconvg} is proven, as well as \eqref{eq:triconv}. We next define the mapping $\chi: [0, \infty]^2 \times (0, \infty] \mapsto [0, \infty]^2 \times (0, \infty]$ 
\begin{align*}
\chi\left((x, y), a \right) = \left((ax, y), a \right).
\end{align*}
This mapping satisfies the conditions in Proposition 5.5 of \cite{resnick2007heavy} and thus we may apply it to the convergence in \eqref{eq:triconv} to obtain that
\begin{align}
\label{eq:triconv2}
t\mathbb{P}\left( \left(\frac{|\mathcal{D}_{11} - \mathcal{D}_{12}|}{\sqrt{2}b_0(t)}, \frac{\sqrt{2}\mathcal{D}_{11}/\sqrt{W_{11}}}{|\mathcal{D}_{11} - \mathcal{D}_{12}|}, \frac{\sqrt{W_{11}}}{b_0(t)} \right) \in \cdot \right)  \xrightarrow{v} (S \times \nu_{2\alpha}) \circ \chi^{-1}(\cdot),
\end{align}
in $ M_+([0, \infty]^2 \times (0, \infty])$. From \eqref{eq:triconv2}, we aim to prove that
\begin{align}
\label{eq:finalconv}
t\mathbb{P}\left( \left(\frac{|\mathcal{D}_{11} - \mathcal{D}_{12}|}{\sqrt{2}b_0(t)}, \frac{\sqrt{2}\mathcal{D}_{11}/\sqrt{W_{11}}}{|\mathcal{D}_{11} - \mathcal{D}_{12}|}\right) \in \cdot \right)  \xrightarrow{v} (S \times \nu_{2\alpha}) \circ \chi^{-1}(\cdot \times (0, \infty]),
\end{align}
in $ M_+((0, \infty] \times [0, \infty])$. Let $f$ be a positive, bounded and continuous function with compact support in $(0, \infty] \times [0, \infty]$. Since the support is compact, we may posit the existence of a $\theta > 0$ such that the compact support of $f$ lies outside of $[0, \theta] \times [0, \infty]$. From \eqref{eq:triconv2}, we then have for any $\tau > 0$
\begin{align*}
t\mathbb{E}\left[f\left( \frac{|\mathcal{D}_{11} - \mathcal{D}_{12}|}{\sqrt{2}b_0(t)}, \frac{\sqrt{2}\mathcal{D}_{11}/\sqrt{W}_{11}}{ |\mathcal{D}_{11} - \mathcal{D}_{12}|} \right)1_{\left\lbrace \sqrt{W_{11}} > b_0(t)\tau \right\rbrace}\right]& \\
\rightarrow \int_{\tau}^\infty &\mathbb{E}\left( f\left(x|Z|, \frac{1}{|Z|}\right)\right)\nu_{2\alpha}(dx).
\end{align*}
Hence, in order to show \eqref{eq:finalconv}, it suffices to show that
\begin{align*}
\lim_{\tau \rightarrow 0}\limsup_{t \rightarrow \infty} t\mathbb{E}\left[f\left( \frac{|\mathcal{D}_{11} - \mathcal{D}_{12}|}{\sqrt{2}b_0(t)}, \frac{\sqrt{2}\mathcal{D}_{11}/\sqrt{W_{11}}}{ |\mathcal{D}_{11} - \mathcal{D}_{12}|} \right)1_{\left\lbrace \sqrt{W_{11}} \leq b_0(t)\tau \right\rbrace}\right] = 0.
\end{align*}
Since $f$ is bounded and supported on the complement of $[0, \theta] \times [0, \infty]$, we have that
\begin{align*}
t\mathbb{E}&\left[f\left( \frac{|\mathcal{D}_{11} - \mathcal{D}_{12}|}{\sqrt{2}b_0(t)}, \frac{\sqrt{2}\mathcal{D}_{11}/\sqrt{W}_{11}}{ |\mathcal{D}_{11} - \mathcal{D}_{12}|} \right)1_{\left\lbrace \sqrt{W_{11}} \leq b_0(t)\tau \right\rbrace}\right] \\
\leq& \sup_{\mathbf{x} \in \mathbb{R}_+^2}|f(\mathbf{x})| t\mathbb{P}\left(\frac{|\mathcal{D}_{11} - \mathcal{D}_{12}|}{\sqrt{2}b_0(t)} > \theta , \sqrt{W_{11}} \leq b_0(t)\tau\right) \\
\equiv& \sup_{\mathbf{x} \in \mathbb{R}_+^2}|f(\mathbf{x})|B(t, \tau).
\end{align*}
Hence, it suffices to show that $\lim_{\tau \rightarrow 0}\limsup_{t \rightarrow \infty}B(t, \tau) =0$. For $B(t, \tau)$, we employ Lemma \ref{lem:mbound} to achieve that for any $m \in \mathbb{N}$
\begin{align*}
B(t, \tau) =&t\mathbb{E}\left[\mathbb{P}\left(\frac{|\mathcal{D}_{11} - \mathcal{D}_{12}|}{\sqrt{2}b_0(t)} > \theta \mid W_{11} \right) 1_{\left\lbrace \sqrt{W_{11}} \leq b_0(t)\tau \right\rbrace}\right] \\
\leq& 2^{m/2}\theta^{-m}\frac{t}{b^m_0(t)}\mathbb{E}\left[(a_mW^{m/2}_{11} + C_m) 1_{\left\lbrace \sqrt{W_{11}} \leq b_0(t)\tau \right\rbrace}\right] \\
\leq& 2^{m/2}\theta^{-m}a_m\frac{t}{b^m_0(t)}\mathbb{E}\left[W^{m/2}_{11} 1_{\left\lbrace \sqrt{W_{11}} \leq b_0(t)\tau \right\rbrace}\right] +  2^{m/2}\theta^{-m}C_m\frac{t}{b^m_0(t)}.
\end{align*}
Choosing $m > 2\alpha$ makes the second hand term in the previous display tend to zero. By Karamata's Theorem, we also achieve that
\begin{align*}
\frac{t}{b^m_0(t)}\mathbb{E}\left[W^{m/2}_{11} 1_{\left\lbrace \sqrt{W_{11}} \leq b_0(t)\tau \right\rbrace}\right] \leq& \frac{t}{b^m_0(t)} \int_0^{b_0^{m}(t)\tau^m}\mathbb{P}\left( W_{11}^{m/2} > x \right) dx \\
\sim& \left(1 - \frac{2\alpha}{m}\right)\tau^m t\mathbb{P}\left( W_{11}^{m/2} > b_0^{m}(t)\tau^m \right) \\
\sim& \left(1 - \frac{2\alpha}{m}\right)\tau^{m-2\alpha}.
\end{align*}
Hence $\lim_{\tau \rightarrow 0}\limsup_{t \rightarrow \infty}B(t, \tau) = 0$ and we have proven \eqref{eq:finalconv}. One may evaluate the convergence \eqref{eq:finalconv} on sets of the form $(u, \infty] \times (v, \infty]$ $u, v \geq 0$ to obtain that as $t \rightarrow \infty$
\begin{align*}
t\mathbb{P}\left( \frac{|\mathcal{D}_{11} - \mathcal{D}_{12}|}{\sqrt{2}b_0(t)} > u, \frac{\sqrt{2}\mathcal{D}_{11}/\sqrt{W_{11}}}{|\mathcal{D}_{11} - \mathcal{D}_{12}|} > v\right) \rightarrow u^{-2\alpha} \cdot 2\int_0^{1/v} z^{2\alpha} \phi(z) dz,
\end{align*}
where $\phi$ is the standard normal density. 
\end{proof}

\subsection{Tail empirical measure approximation}\label{sec:tempapprox}

In this section we present results that allow us to approximate the tail empirical measure of the degrees \eqref{eq:tailempD} by the tail empirical measure of the weights \eqref{eq:tailempW}. From \eqref{eq:degreecondNR}, we expect that the large degrees in layers $\mathcal{L}_1$ will concentrate around their conditional mean due to the strong concentration of Poisson distributions with large rates. Additionally, we expect the conditional means to concentrate around the scaled weights, as indicated by Lemma \ref{lem:gconv}. We also expect results derived from the layers in $\mathcal{L}_1$ to translate to layers in $\mathcal{L}_2$ using the coupling developed in Section \ref{sec:couple}. Thus, our proof strategy proceeds as follows.
\begin{itemize}
\item[1.] In Section \ref{sec:degconc}, we prove concentration of the degrees around the conditional means for two subsets of models:
\begin{itemize}
\item[(i)]  $|\mathcal{L}_1| = L$ and
\item[(ii)]  $|\mathcal{L}_1| < L, |\mathcal{L}_2| \leq L$, $\alpha \in (1, 2]$.
\end{itemize}
We are not able to use the coupling in Section \ref{sec:couple} to handle case (ii) since it is seemingly inapplicable in the infinite-variance setting. Such limitations also appear in \cite{bhattacharjee2022large} (see Section 4.2). 
\item[2.] In Section \ref{sec:gconc}, we prove concentration of the conditional mean degree around the scaled weights. Note that we need not distinguish between models (i) and (ii) in this case.
\item[3.] In Section \ref{sec:tempproof}, we connect results developed in Sections \ref{sec:degconc} and  \ref{sec:gconc} to produce the tail empirical approximation for submodels (i) and (ii). We then use the coupling developed in Section \ref{sec:couple} to prove the tail empirical measure approximation for the submodel $|\mathcal{L}_1| < L, |\mathcal{L}_2| \leq L$, $\alpha > 2$ by relating it to (i).
\end{itemize}

\subsubsection{Concentration of the degrees}\label{sec:degconc}

Lemma \ref{lem:Dapprox} considers concentration of the degrees around their conditional means when $|\mathcal{L}_1| = L$. Define the conditional mean $\boldsymbol{g}_1(n) = \mathbb{E}[\mathbf{D}_1(n) \mid \mathbf{W}_{[n]}]$ and recall $\boldsymbol{\mathcal{W}}_1 = (c_1W_{11}, \dots, c_LW_{1L})$.

\begin{lemma}
\label{lem:Dapprox}
Suppose $|\mathcal{L}_1| = L$. Assume (C1)-(C3) hold. Let $\eta$ and $u$ be strictly positive real numbers. As $n \rightarrow \infty$
\begin{align}
\label{eq:Dapprox1}
&\frac{n}{k_n}  \mathbb{P}\left(\| \mathbf{D}_1(n) - \boldsymbol{g}_1(n)  \|_p  \geq \eta b(n/k_n), \| \boldsymbol{\mathcal{W}}_1 \|_p \leq ub(n/k_n)\right) \rightarrow 0, \\
\label{eq:Dapprox2}
&\frac{n}{k_n}  \mathbb{P}\left(\| \mathbf{D}_1(n) - \boldsymbol{g}_1(n)  \|_p  \geq \eta b(n/k_n), \| \boldsymbol{\mathcal{W}}_1 \|_p \geq ub(n/k_n)\right)\rightarrow 0.
\end{align}
\end{lemma}

\begin{proof}
Fix $\epsilon \in (0, \delta)$, where $\delta$ is as in \eqref{eq:glin}. Define the event
\begin{align*}
\mathcal{A}_n(\epsilon) = \left\lbrace \max_{l \in [L]}  W_{1l}\frac{W_{(1)l}(n)}{T_l(n)} < \epsilon \right\rbrace.
\end{align*}
Note that due to Lemma \ref{lem:Wsmall}, in suffices to prove that
\begin{align}
\label{eq:Dapprox12}
&\frac{n}{k_n}  \mathbb{P}\left(\| \mathbf{D}_1(n) - \boldsymbol{g}_1(n)  \|_p  \geq \eta b(n/k_n), \| \boldsymbol{\mathcal{W}}_1 \|_p \leq ub(n/k_n), \mathcal{A}_n(\epsilon) \right) \rightarrow 0, \\
\label{eq:Dapprox22}
&\frac{n}{k_n}  \mathbb{P}\left(\| \mathbf{D}_1(n) - \boldsymbol{g}_1(n)  \|_p  \geq \eta b(n/k_n), \| \boldsymbol{\mathcal{W}}_1 \|_p \geq ub(n/k_n), \mathcal{A}_n(\epsilon) \right)\rightarrow 0.
\end{align}
We start by proving \eqref{eq:Dapprox12}. Let $m$ be a positive integer such that $m > 2\alpha$. We may apply Markov's inequality and Lemma \ref{lem:mnorm} to achieve that
\begin{align*}
&\frac{n}{k_n}  \mathbb{P}\left(\| \mathbf{D}_1(n) - \boldsymbol{g}_1(n)  \|_p  \geq \eta b(n/k_n), \| \boldsymbol{\mathcal{W}}_1 \|_p \leq ub(n/k_n), \mathcal{A}_n(\epsilon) \right) \\
&\leq \left(\frac{1}{\eta b(n/k_n)}\right)^m \frac{n}{k_n} \mathbb{E}\left[\mathbb{E}\left[\| \mathbf{D}_1(n) - \boldsymbol{g}_1(n)  \|^m_p \mid \boldsymbol{\mathcal{W}}_{[n]} \right] 1_{\left\lbrace \| \boldsymbol{\mathcal{W}}_1 \|_p \leq ub(n/k_n), \mathcal{A}_n(\epsilon) \right\rbrace} \right] \\
&\leq \left(\frac{L^{\frac{1}{p} + \frac{1}{m}}}{\eta b(n/k_n)}\right)^m  \frac{n}{k_n} \mathbb{E}\left[\left(a_{m} \| \boldsymbol{g}_1(n) \|^{m/2}_{p}  + C_{m}\right)1_{\left\lbrace \| \boldsymbol{\mathcal{W}}_1 \|_p \leq ub(n/k_n), \mathcal{A}_n(\epsilon)  \right\rbrace} \right]. 
\end{align*}
Since $m > \alpha$, we may determine that as $n \rightarrow \infty$
\begin{align*}
\left(\frac{1}{b(n/k_n)}\right)^m\frac{n}{k_n}\mathbb{P}\left(  \| \boldsymbol{\mathcal{W}}_1 \|_p \leq ub(n/k_n) ,  \mathcal{A}_n(\epsilon) \right) \leq \left(\frac{1}{b(n/k_n)}\right)^m \frac{n}{k_n} \rightarrow 0.
\end{align*}
Hence in order to prove \eqref{eq:Dapprox12}, it suffices to show that 
\begin{align*}
H_1(n) \equiv \left(\frac{1}{b(n/k_n)}\right)^m \frac{n}{k_n} \mathbb{E}\left[\| \boldsymbol{g}_1(n)\|_{p}^{m/2}1_{\left\lbrace \| \boldsymbol{\mathcal{W}}_1 \|_p \leq ub(n/k_n), \mathcal{A}_n(\epsilon)  \right\rbrace} \right] \rightarrow 0, 
\end{align*}
as $n \rightarrow \infty$. Note that when $\max_{l \in [L]}  W_{1l}\frac{W_{(1)l}(n)}{T_l(n)} < \epsilon < \delta$,
\begin{align*}
\sum_{l = 1}^L \left(\sum_{j = 1}^n g_l\left( W_{1l}W_{jl}/T_l(n) \right) \right)^{p} \leq& \sum_{l = 1}^L\left(M W^{1 +\nu}_{1l}\sum_{j = 1}^n\frac{W^{1 +\nu}_{jl}}{T^{1 +\nu}_l(n)} + c_lW_{1l}  \right)^{p} \\
\leq& \sum_{l = 1}^L \left(M W^{1 +\nu}_{1l} \frac{W^\nu_{(1)l}(n)}{T^\nu_l(n)} + c_lW_{1l}  \right)^{p} \\
\leq& \sum_{l = 1}^L \left(M \epsilon^\nu W_{1l} + c_lW_{1l}  \right)^{p} \\
\leq& 2^{p} \|\boldsymbol{\mathcal{W}}_1 \|^p_{p},
\end{align*}
where, if necessary, we have taken $\epsilon$ small enough so that $M\epsilon^\nu < \min_{l \in [L]} c_l$. Hence $\| \boldsymbol{g}_1(n)\|_{p}^{m/2} \leq 2^{m/2} \|\boldsymbol{\mathcal{W}}_1 \|^{m/2}_{p}$ and we may write
\begin{align*}
H_1(n) \leq& \left( \frac{2^{1/2}}{b(n/k_n)}\right)^m \frac{n}{k_n} \mathbb{E}\left[\| \boldsymbol{\mathcal{W}}_1\|_{p}^{m/2}1_{\left\lbrace \| \boldsymbol{\mathcal{W}}_1 \|_p \leq ub(n/k_n)\right\rbrace} \right]  \\
\leq&  \left( \frac{2^{1/2}}{b(n/k_n)}\right)^m \frac{n}{k_n}  \int_0^{(ub(n/k_n))^{m/2}} \mathbb{P}\left( \| \boldsymbol{\mathcal{W}}_1 \|^{m/2}_{p} > t\right) dt,
\intertext{and recalling that $m > 2\alpha$ we may apply Karamata's Theorem \citep[Theorem 2.1]{resnick2007heavy} to achieve that}
\sim& \left( \frac{2^{1/2}}{b(n/k_n)}\right)^m \frac{n}{k_n}\left( 1 - \frac{2\alpha}{m}\right)^{-1} \left(ub(n/k_n)\right)^{m/2} \mathbb{P}\left( \| \boldsymbol{\mathcal{W}}_1 \|_{p} >ub(n/k_n) \right) \\
\sim& \left( \frac{2}{b(n/k_n)}\right)^{m/2} \left( 1 - \frac{2\alpha}{m}\right)^{-1} u^{\frac{m}{2} - \alpha} \rightarrow 0,
\end{align*}
as $n \rightarrow \infty$. Hence \eqref{eq:Dapprox12} and is proven and thus so is \eqref{eq:Dapprox1}. We now prove \eqref{eq:Dapprox2}, for which it suffices to prove \eqref{eq:Dapprox22}. By similar calculations as before
\begin{align*}
&\frac{n}{k_n} \mathbb{P}\left(\| \mathbf{D}_1(n) - \boldsymbol{g}_1(n)  \|_p  \geq \eta b(n/k_n), \| \boldsymbol{\mathcal{W}}_1 \|_p \geq ub(n/k_n), \mathcal{A}_n(\epsilon) \right) \\
&\leq \left(\frac{1}{\eta b(n/k_n)}\right)^2 \frac{n}{k_n} \mathbb{E}\left[\mathbb{E}\left[\| \mathbf{D}_1(n) - \boldsymbol{g}_1(n)  \|^2_p \mid \boldsymbol{\mathcal{W}}_{[n]} \right] 1_{\left\lbrace \| \boldsymbol{\mathcal{W}}_1 \|_p \geq ub(n/k_n), \mathcal{A}_n(\epsilon) \right\rbrace} \right] \\
&\leq  \left(\frac{L^{\frac{1}{p} + \frac{1}{2}}}{\eta b(n/k_n)}\right)^2 \frac{n}{k_n} \mathbb{E}\left[\left( a_{2} \| \boldsymbol{g}_1(n)\|_{p} + C_{2}\right)1_{\left\lbrace \| \boldsymbol{\mathcal{W}}_1 \|_p \geq ub(n/k_n), \mathcal{A}_n(\epsilon) \right\rbrace} \right]. 
\end{align*}
For the second summand, note that
\begin{align*}
\left(\frac{1}{b(n/k_n)}\right)^2 \frac{n}{k_n} \mathbb{P}&\left(\| \boldsymbol{\mathcal{W}}_1 \|_p \geq ub(n/k_n), \mathcal{A}_n(\epsilon)  \right) \\
&\leq \left(\frac{1}{b(n/k_n)}\right)^2 \frac{n}{k_n} \mathbb{P}\left(\| \boldsymbol{\mathcal{W}}_1 \|_p \geq ub(n/k_n) \right) \rightarrow 0 \cdot u^{-\alpha} = 0,
\end{align*}
as $n \rightarrow \infty$, by regular variation of $\| \boldsymbol{\mathcal{W}}_1 \|_p$. Hence in order to prove \eqref{eq:Dapprox22}, it suffices to show that 
\begin{align*}
H_2(n) \equiv \left(\frac{1}{b(n/k_n)}\right)^2 \frac{n}{k_n} \mathbb{E}\left[\| \boldsymbol{g}_1(n)\|_{p}1_{\left\lbrace \| \boldsymbol{\mathcal{W}}_1 \|_p \geq ub(n/k_n), \mathcal{A}_n(\epsilon)  \right\rbrace} \right] \rightarrow 0, 
\end{align*}
as $n \rightarrow \infty$. Recall that $\max_{l \in [L]}  W_{1l}\frac{W_{(1)l}(n)}{T_l(n)} < \epsilon < \delta$ implies that $\| \boldsymbol{g}_1(n)\|_{p} \leq 2 \|\boldsymbol{\mathcal{W}}_1 \|_{p}$ for $\epsilon$ small enough and thus
\begin{align*}
H_2(n) \leq& 2\left(\frac{1}{b(n/k_n)}\right)^2 \frac{n}{k_n} \mathbb{E}\left[\|\boldsymbol{\mathcal{W}}_1 \|_{p}1_{\left\lbrace \| \boldsymbol{\mathcal{W}}_1 \|_p \geq ub(n/k_n) \right\rbrace} \right], \\
=& 2\left(\frac{1}{b(n/k_n)}\right)^2 \frac{n}{k_n} \int_0^\infty \mathbb{P}\left( \|\boldsymbol{\mathcal{W}}_1 \|_{p}1_{\left\lbrace \| \boldsymbol{\mathcal{W}}_1 \|_p \geq ub(n/k_n)\right\rbrace} > t \right)  dt \\
=&2\left(\frac{1}{b(n/k_n)}\right)^2 \frac{n}{k_n} ub(n/k_n)\mathbb{P}\left( \| \boldsymbol{\mathcal{W}}_1 \|_p \geq ub(n/k_n) \right) \\
&+ 2\left(\frac{1}{b(n/k_n)}\right)^2 \frac{n}{k_n} \int_{ub(n/k_n)}^\infty \mathbb{P}\left( \|\boldsymbol{\mathcal{W}}_1 \|_{p}  > t \right) dt \\
\sim& 2\frac{1}{b(n/k_n)}u^{1 - \alpha} + 2\left(\frac{1}{b(n/k_n)}\right)^2  \frac{n}{k_n}(\alpha - 1)^{-1}ub(n/k_n)\mathbb{P}\left( \|\boldsymbol{\mathcal{W}}_1 \|_{p}  > ub(n/k_n) \right) \\
\sim& 2\frac{1}{b(n/k_n)}u^{1 - \alpha}  + 2\frac{1}{b(n/k_n)}(\alpha - 1)^{-1}u^{1 - \alpha},
\end{align*}
where we have again applied Karamata's Theorem since $\|\boldsymbol{\mathcal{W}}_1 \|_{p}$ has a tail index of $\alpha > 1$. Hence $H_2(n) \rightarrow 0$ as $n \rightarrow \infty$ and \eqref{eq:Dapprox22} is proven. Thus we have proved \eqref{eq:Dapprox2}.
\end{proof}

Lemma \ref{lem:Dapprox2} considers concentration of the degrees around their conditional means when $|\mathcal{L}_1| < L, |\mathcal{L}_2| \leq L$ and $\alpha \in (1, 2]$.

\begin{lemma}
\label{lem:Dapprox2}
Suppose $|\mathcal{L}_1| < L, |\mathcal{L}_2| \leq L$ and that (C1), (C3) hold. Assume (C2) holds with the restriction that $\alpha \in (1, 2]$. Then
\begin{align*}
\frac{n}{k_n}  \mathbb{P}\left(\| \mathbf{D}_1(n) - \boldsymbol{g}_1(n)  \|_p  \geq \eta b(n/k_n) \right) \rightarrow 0,
\end{align*}
as $n \rightarrow \infty$.
\end{lemma}

\begin{proof}
Let $\delta$ be as in \eqref{eq:glin}. in From Lemma \ref{lem:Wsmall}, it suffies to show that
\begin{align*}
\frac{n}{k_n}  \mathbb{P}\left(\| \mathbf{D}_1(n) - \boldsymbol{g}_1(n)  \|_p  \geq \eta b(n/k_n), \max_{l \in [L]} W_{1l}\frac{W_{(1)l}(n)}{T_l(n)} < \delta \right) \rightarrow 0,
\end{align*}
as $n \rightarrow \infty$.
Suppose $\alpha \in (1, 2)$. By equivalence of $\ell_p$ norms, we have that
\begin{align*}
\frac{n}{k_n} & \mathbb{P}\left(\| \mathbf{D}_1(n) - \boldsymbol{g}_1(n)  \|_p  \geq \eta b(n/k_n), \max_{l \in [L]} W_{1l}\frac{W_{(1)l}(n)}{T_l(n)} < \delta \right) \\
\leq& \frac{n}{k_n}  \mathbb{P}\left(\| \mathbf{D}_1(n) - \boldsymbol{g}_1(n)  \|_2  \geq \eta b(n/k_n)/L^{1/p}, \max_{l \in [L]} W_{1l}\frac{W_{(1)l}(n)}{T_l(n)} < \delta\right) \\
=&\frac{n}{k_n} \mathbb{E}\left[ \mathbb{P}\left(\| \mathbf{D}_1(n) - \boldsymbol{g}_1(n)  \|_2  \geq \eta b(n/k_n)/L^{1/p} \  \bigg\vert \ \boldsymbol{\mathcal{W}}_{[n]} \right)1_{\left\lbrace  \max_{l \in [L]} W_{1l}\frac{W_{(1)l}(n)}{T_l(n)} < \delta \right\rbrace}\right] \\
\leq& \frac{n}{k_n} \left(\frac{L^{1/p}}{\eta b(n/k_n)} \right)^2 \mathbb{E}\left[ \mathbb{E}\left[\| \mathbf{D}_1(n) - \boldsymbol{g}_1(n)  \|^2_2 \mid \boldsymbol{\mathcal{W}}_{[n]} \right]1_{\left\lbrace  \max_{l \in [L]} W_{1l}\frac{W_{(1)l}(n)}{T_l(n)} < \delta \right\rbrace} \right].
\end{align*}
Note that for $l \in \mathcal{L}_1$
\begin{align*}
\mathbb{E}\left[ \left(D_{1l}(n) - \sum_{j = 1}^n g_l(W_{1l}W_{jl}/T_l(n)) \right)^2  \  \Bigg\vert \ \boldsymbol{\mathcal{W}}_{[n]} \right] = \sum_{j = 1}^n g_l(W_{1l}W_{jl}/T_l(n)),
\end{align*}
and for $l \in \mathcal{L}_2$
\begin{align*}
\mathbb{E}&\left[ \left(D_{1l}(n) - \sum_{j = 1}^n g_l(W_{1l}W_{jl}/T_l(n)) \right)^2  \  \Bigg\vert \ \boldsymbol{\mathcal{W}}_{[n]} \right] \\
=& \sum_{j = 1}^n g_l(W_{1l}W_{jl}/T_l(n))(1 - g_l(W_{1l}W_{jl}/T_l(n))) \\
\leq& \sum_{j = 1}^n g_l(W_{1l}W_{jl}/T_l(n)).
\end{align*}
Additionally, when $\max_{l \in [L]} W_{1l}\frac{W_{(1)l}(n)}{T_l(n)} < \delta$, 
\begin{align*}
\sum_{j = 1}^n g_l(W_{1l}W_{jl}/T_l(n)) \leq& c_l W_{1l} + M \sum_{j = 1}^n \left( \frac{W_{1l}W_{jl}}{T_l(n)} \right)^{1 + \nu} \\
\leq& c_l W_{1l} + M W^{1 + \nu}_{1l} \left( \frac{W_{(1)l}(n)}{T_l(n)} \right)^{\nu} \\
\leq& \left(c_l + M \delta^\nu \right)W_{1l}.
\end{align*}
Hence
\begin{align*}
\frac{n}{k_n} & \mathbb{P}\left(\| \mathbf{D}_1(n) - \boldsymbol{g}_1(n)  \|_p  \geq \eta b(n/k_n), \max_{l \in [L]} W_{1l}\frac{W_{(1)l}(n)}{T_l(n)} < \delta \right) \\
\leq& \frac{n}{k_n} \left(\frac{L^{1/p}}{\eta b(n/k_n)} \right)^2 \mathbb{E}\left[ \mathbb{E}\left[\| \mathbf{D}_1(n) - \boldsymbol{g}_1(n)  \|^2_2 \mid \boldsymbol{\mathcal{W}}_{[n]} \right]1_{\left\lbrace  \max_{l \in [L]} W_{1l}\frac{W_{(1)l}(n)}{T_l(n)} < \delta \right\rbrace} \right] \\
\leq& \frac{n}{k_n} \left(\frac{L^{1/p}}{\eta b(n/k_n)} \right)^2 \sum_{l = 1}^L\mathbb{E}\left[  \left(c_l + M \delta^\nu \right)W_{1l} \right].
\end{align*}
Since $-\alpha < -1$, $\mathbb{E}[W_{1l}] < \infty$ and since $b(t)$ is regularly varying with index $1/\alpha$, we have that  
\begin{align*}
\frac{n}{k_n} & \mathbb{P}\left(\| \mathbf{D}_1(n) - \boldsymbol{g}_1(n)  \|_p  \geq \eta b(n/k_n), \max_{l \in [L]} W_{1l}\frac{W_{(1)l}(n)}{T_l(n)} < \delta \right) \rightarrow 0 \qquad \text{as } n \rightarrow \infty.
\end{align*}
Now suppose $\alpha = 2$. Using similar steps we find that
\begin{align*}
\frac{n}{k_n} & \mathbb{P}\left(\| \mathbf{D}_1(n) - \boldsymbol{g}_1(n)  \|_p  \geq \eta b(n/k_n), \max_{l \in [L]} W_{1l}\frac{W_{(1)l}(n)}{T_l(n)} < \delta \right) \\
\leq& \frac{n}{k_n} \left(\frac{L^{1/p}}{\eta b(n/k_n)} \right)^3 \mathbb{E}\left[ \mathbb{E}\left[\| \mathbf{D}_1(n) - \boldsymbol{g}_1(n)  \|^3_3 \mid \boldsymbol{\mathcal{W}}_{[n]} \right]1_{\left\lbrace  \max_{l \in [L]} W_{1l}\frac{W_{(1)l}(n)}{T_l(n)} < \delta \right\rbrace} \right]
\end{align*}
Note that for $l \in \mathcal{L}_1$, by Lemma \ref{lem:mbound}
\begin{align*}
\mathbb{E}\left[ \left| D_{1l}(n) - \sum_{j = 1}^n g_l(W_{1l}W_{jl}/T_l(n)) \right|^3 \ \Bigg\vert \ \boldsymbol{\mathcal{W}}_{[n]} \right] \leq a_3 \left( \sum_{j = 1}^n g_l(W_{1l}W_{jl}/T_l(n))\right)^{3/2} + C_3.
\end{align*}
Applying a bound for the absolute third central moment of Poisson binomial random variables developed in Lemma \ref{lem:PB3}, we further find that for $l \in \mathcal{L}_2$
\begin{align*}
\mathbb{E}\left[ \left| D_{1l}(n) - \sum_{j = 1}^n g_l(W_{1l}W_{jl}/T_l(n)) \right|^3 \ \Bigg\vert \ \boldsymbol{\mathcal{W}}_{[n]} \right] \leq&  2\sum_{j = 1}^n g(W_{1l}W_{jl}/T_l(n)) \\
&+ 2\left(\sum_{j = 1}^n g(W_{1l}W_{jl}/T_l(n))\right)^{3/2}.
\end{align*}
When $\max_{l \in [L]} W_{1l}\frac{W_{(1)l}(n)}{T_l(n)} < \delta$, 
\begin{align*}
\sum_{j = 1}^n g_l(W_{1l}W_{jl}/T_l(n))\leq \left(c_l + M \delta^\nu \right)W_{1l}, 
\end{align*}
and hence
\begin{align*}
\frac{n}{k_n} & \mathbb{P}\left(\| \mathbf{D}_1(n) - \boldsymbol{g}_1(n)  \|_p  \geq \eta b(n/k_n), \max_{l \in [L]} W_{1l}\frac{W_{(1)l}(n)}{T_l(n)} < \delta \right) \\
\leq& \frac{n}{k_n} \left(\frac{L^{1/p}}{\eta b(n/k_n)} \right)^3 \mathbb{E}\left[ \mathbb{E}\left[\| \mathbf{D}_1(n) - \boldsymbol{g}_1(n)  \|^3_3 \mid \boldsymbol{\mathcal{W}}_{[n]} \right]1_{\left\lbrace  \max_{l \in [L]} W_{1l}\frac{W_{(1)l}(n)}{T_l(n)} < \delta \right\rbrace} \right] \\
\leq& \frac{n}{k_n} \left(\frac{L^{1/p}}{\eta b(n/k_n)} \right)^3 \sum_{l \in \mathcal{L}_1}\left(a_3  \left(c_l + M \delta^\nu \right)^{3/2}\mathbb{E}\left[ W^{3/2}_{1l}\right] + C_3 \right) \\
&+ 2\frac{n}{k_n} \left(\frac{L^{1/p}}{\eta b(n/k_n)} \right)^3 \sum_{l \in \mathcal{L}_2}\left(\left(c_l + M \delta^\nu \right)\mathbb{E}\left[W_{1l}\right] + \left(c_l + M \delta^\nu \right)^{3/2}\mathbb{E}\left[ W^{3/2}_{1l}\right]  \right).
\end{align*}
Since $\alpha = 2$, $\mathbb{E}\left[ W_{1l}\right], \mathbb{E}\left[ W^{3/2}_{1l}\right] < \infty$. Additionally, since $b(t)$ is regularly varying with index $1/2$, we have that  
\begin{align*}
\frac{n}{k_n} \mathbb{P}\left(\| \mathbf{D}_1(n) - \boldsymbol{g}_1(n)  \|_p  \geq \eta b(n/k_n), \max_{l \in [L]} W_{1l}\frac{W_{(1)l}(n)}{T_l(n)} < \delta \right) \rightarrow 0, \qquad \text{as }n \rightarrow \infty.
\end{align*}
\end{proof}

\begin{lemma}
\label{lem:PB3}
Fix $n \in \mathbb{N}$. Suppose $X_1, X_2, \dots, X_n$ are independent Bernoulli random variables with success probabilities $p_1, p_2, \dots, p_n$, respectively. Then 
\begin{align}
\label{eq:PB3}
\mathbb{E}\left[\left|\sum_{i = 1}^n (X_i - p_i) \right|^3\right] \leq 2\sum_{i = 1}^n p_i + 2\left(\sum_{i = 1}^n p_i\right)^{3/2}.
\end{align}
\end{lemma}

\begin{proof}
A tedious calculation gives that
\begin{align*}
\mathbb{E}\left[\left(\sum_{i = 1}^n (X_i - p_i) \right)^4\right] =& \sum_{i = 1}^n p_i(1 - p_i)\left((1 - 2p_i)^2 - 2p_i(1 - p_i) \right) + 3\left( \sum_{i = 1}^np_i(1 - p_i) \right)^2 \\
\leq& \sum_{i = 1}^n p_i + 3 \left(\sum_{i = 1}^n p_i\right)^2.
\end{align*}
Hence by Cauchy-Bunyakovsky-Schwarz,
\begin{align*}
\mathbb{E}\left[\left|\sum_{i = 1}^n (X_i - p_i) \right|^3\right] \leq& \left( \mathbb{E}\left[\left(\sum_{i = 1}^n (X_i - p_i) \right)^4\right]\mathbb{E}\left[\left(\sum_{i = 1}^n (X_i - p_i) \right)^2\right]\right)^{1/2}, \\
\intertext{and using the fact that $\mathbb{E}\left[\left(\sum_{i = 1}^n (X_i - p_i) \right)^2\right] = \sum_{i = 1}^n p_i(1-p_i) \leq \sum_{i = 1}^n p_i$, }
\leq& \left( \left(\sum_{i = 1}^n p_i + 3 \left(\sum_{i = 1}^n p_i\right)^2\right) \sum_{i = 1}^n p_i \right)^{1/2} \\
=& \left( \left(\sum_{i = 1}^n p_i\right)^2 + 3 \left(\sum_{i = 1}^n p_i\right)^3\right)^{1/2} \\
\leq& 2\max\left\lbrace \sum_{i = 1}^n p_i, \left(\sum_{i = 1}^n p_i\right)^{3/2} \right\rbrace \\
\leq& 2\sum_{i = 1}^n p_i + 2\left(\sum_{i = 1}^n p_i\right)^{3/2}.
\end{align*}
\end{proof}

\subsubsection{Concentration of the conditional mean}\label{sec:gconc}

Lemma \ref{lem:gapprox} considers concentration of the conditional mean around the scaled weights. 

\begin{lemma}
\label{lem:gapprox}
Suppose (C1)-(C3) hold. For any $\eta > 0,$ and $p \in \mathbb{N}$
\begin{align}
\label{eq:gapprox}
&\frac{n}{k_n} \mathbb{P}\left(\|\boldsymbol{g}_1(n) - \boldsymbol{\mathcal{W}}_1 \|_p > \eta b(n/k_n) \right) \rightarrow 0, 
\end{align}
as $n \rightarrow \infty$.
\end{lemma}

\begin{proof}
Fix $\epsilon \in (0, \delta)$, where $\delta$ is as in \eqref{eq:glin}. If $\max_{l \in [L]}  W_{1l}\frac{W_{(1)l}(n)}{T_l(n)} < \epsilon < \delta$, we find that
\begin{align*}
\|\boldsymbol{g}_1(n) - \boldsymbol{\mathcal{W}}_1 \|_p \leq& \left( M^p \sum_{l = 1}^L \left( W^{1 + \nu}_{1l} \sum_{j = 1}^n \frac{W^{1 + \nu}_{jl}}{T^{1 + \nu}_l(n)} \right)^p   \right)^{\frac{1}{p}} \\
\leq& M \left( \sum_{l = 1}^L \left( W^{1 + \nu}_{1l} \frac{ W^{\nu}_{(1)l}(n)}{T^{\nu}_l(n)} \right)^p   \right)^{\frac{1}{p}} \\
\leq& \epsilon^\nu M \left( \sum_{l = 1}^L  W^p_{1l}  \right)^{\frac{1}{p}} \\ 
\leq& \frac{\epsilon^\nu M}{c_\text{min}^p} \left( \sum_{l = 1}^L  c_l^p W^p_{1l}  \right)^{\frac{1}{p}} \\
=& \frac{\epsilon^\nu M}{c_\text{min}^p}\|\boldsymbol{\mathcal{W}}_1 \|_p,
\end{align*}
where we define $c_\text{min} = \min \{c_1, \dots, c_L \}$. Hence
\begin{align*}
\frac{n}{k_n} \mathbb{P}&\left(\|\boldsymbol{g}_1(n) - \boldsymbol{\mathcal{W}}_1 \|_p > \eta b(n/k_n) \right) \\
\leq& \frac{n}{k_n} \mathbb{P}\left(\|\boldsymbol{g}_1(n) - \boldsymbol{\mathcal{W}}_1 \|_p > \eta b(n/k_n), \max_{l \in [L]}  W_{1l}\frac{W_{(1)l}(n)}{T_l(n)} < \epsilon \right) \\
&+ \frac{n}{k_n} \mathbb{P}\left(\max_{l \in [L]}  W_{1l}\frac{W_{(1)l}(n)}{T_l(n)} > \epsilon \right) \\
\leq& \frac{n}{k_n} \mathbb{P}\left(\|\boldsymbol{\mathcal{W}}_1 \|_p > \frac{\eta c_\text{min}^p}{\epsilon^\nu M} b(n/k_n) \right) + \frac{n}{k_n} \sum_{l = 1}^L\mathbb{P}\left(W_{1l}\frac{W_{(1)l}(n)}{T_l(n)} > \epsilon \right) \\
\equiv& C_1(n) + C_2(n).
\end{align*}
From Lemma \ref{lem:Wsmall} we have that $C_2(n) \rightarrow 0$ as $n \rightarrow \infty$. Note that due to regular variation of $\boldsymbol{\mathcal{W}}_1$ (and hence $\|\boldsymbol{\mathcal{W}}_1\|_p$), 
\begin{align*}
C_1(n) = \frac{n}{k_n} \mathbb{P}\left(\|\boldsymbol{\mathcal{W}}_1 \|_p > \frac{\eta c^p_\text{min}}{\epsilon^\nu M} b(n/k_n) \right) \rightarrow \left(\frac{\epsilon^\nu M}{\eta c^p_\text{min}}\right)^{\alpha}, \quad \text{as } n \rightarrow \infty.  
\end{align*}
Since $\epsilon > 0$ was arbitrary, we may conclude that $C_1(n) \rightarrow 0$ as $n \rightarrow \infty$. 
\end{proof}

\subsubsection{Proof of tail empirical measure approximation}\label{sec:tempproof}

This subsection is devoted to the proof of Theorem \ref{thm:weakconvD}. In order to prove Theorem \ref{thm:weakconvD}, we require the following lemma.

\begin{lemma}
\label{lem:tailempapprox}
Suppose (C1)-(C3) hold. Then as $n \rightarrow \infty$
\begin{align}
\label{eq:lemtailemp1}
&\frac{n}{k_n} \mathbb{P}\left( \| \mathbf{D}_1(n) \|_p > yb(n/k_n),  \| \boldsymbol{\mathcal{W}}_1 \|_p \leq yb(n/k_n) \right) \rightarrow 0, \\
\label{eq:lemtailemp2}
&\frac{n}{k_n} \mathbb{P}\left( \| \boldsymbol{\mathcal{W}}_1 \|_p > yb(n/k_n),  \| \mathbf{D}_1(n) \|_p \leq yb(n/k_n) \right) \rightarrow 0.
\end{align}
\end{lemma}

We divide the proof of Lemma \ref{lem:tailempapprox} into two parts. The first, presented in Lemma \ref{lem:tailempapproxsub}, employs concentration results developed in Sections \ref{sec:degconc} and \ref{sec:gconc} to prove a tail empirical measure approximation for the submodels where either $|\mathcal{L}_1| = L$ or $|\mathcal{L}_1| < L, |\mathcal{L}_2| \leq L$ and $\alpha \in (1, 2]$. The second, presented in Lemma \ref{lem:coupletempfinal}, employs the coupling presented in Section \ref{sec:couple} to approximate the tail empirical measure when $|\mathcal{L}_1| < L, |\mathcal{L}_2| \leq L$ and $\alpha > 2$. 

\begin{lemma}
\label{lem:tailempapproxsub}
Suppose (C1)-(C3) hold. Additionally, suppose either  $|\mathcal{L}_1| = L$ or $|\mathcal{L}_1| < L, |\mathcal{L}_2| \leq L$ with the restriction that $\alpha \in (1, 2]$. Then as $n \rightarrow \infty$
\begin{align}
\label{eq:tailemp1}
&\frac{n}{k_n} \mathbb{P}\left( \| \mathbf{D}_1(n) \|_p > yb(n/k_n),  \| \boldsymbol{\mathcal{W}}_1 \|_p \leq yb(n/k_n) \right) \rightarrow 0, \\
\label{eq:tailemp2}
&\frac{n}{k_n} \mathbb{P}\left( \| \boldsymbol{\mathcal{W}}_1 \|_p > yb(n/k_n),  \| \mathbf{D}_1(n) \|_p \leq yb(n/k_n) \right) \rightarrow 0.
\end{align}
\end{lemma}

\begin{proof}
Fix $\eta > 0$. See that
\begin{align*}
\frac{n}{k_n} \mathbb{P}&\left( \| \mathbf{D}_1(n) \|_p > yb(n/k_n),  \| \boldsymbol{\mathcal{W}}_1\|_p \leq yb(n/k_n) \right) \\
\leq& \frac{n}{k_n} \mathbb{P}\left(\| \boldsymbol{\mathcal{W}}_1 \|_p \in (y - \eta, y]b(n/k_n) \right) \\
&+ \frac{n}{k_n} \mathbb{P}\left(\| \mathbf{D}_1(n) \|_p - \| \boldsymbol{\mathcal{W}}_1 \|_p \geq \eta b(n/k_n), \| \boldsymbol{\mathcal{W}}_1 \|_p \leq (y-\eta)b(n/k_n) \right) \\
\equiv& A_{11}(n) + A_{12}(n).
\end{align*}
Similarly,
\begin{align*}
\frac{n}{k_n} \mathbb{P}&\left(\| \boldsymbol{\mathcal{W}}_1\|_p  > yb(n/k_n), \| \mathbf{D}_1(n) \|_p \leq yb(n/k_n) \right) \\
\leq& \frac{n}{k_n} \mathbb{P}\left(\| \boldsymbol{\mathcal{W}}_1 \|_p \in (y, y + \eta]b(n/k_n) \right) \\
&+ \frac{n}{k_n} \mathbb{P}\left(\| \boldsymbol{\mathcal{W}}_1 \|_p - \| \mathbf{D}_1(n) \|_p \geq \eta b(n/k_n), \| \boldsymbol{\mathcal{W}}_1 \|_p \geq (y+\eta)b(n/k_n) \right) \\
\equiv& A_{21}(n) + A_{22}(n).
\end{align*}
By regular variation of $\boldsymbol{\mathcal{W}}_1$, we have that as $n \rightarrow \infty$
\begin{align*}
A_{11}(n) =& \frac{n}{k_n} \mathbb{P}\left(\|\boldsymbol{\mathcal{W}}_1 \|_p \in (y - \eta, y]b(n/k_n) \right) \rightarrow (y - \eta)^{-\alpha} - y^{- \alpha}, \\
A_{21}(n) =& \frac{n}{k_n} \mathbb{P}\left(\|\boldsymbol{\mathcal{W}}_1 \|_p \in (y, y + \eta]b(n/k_n) \right) \rightarrow y^{-\alpha} - (y+\eta)^{- \alpha}.
\end{align*}
We now analye the behavior of $A_{12}(n)$. Using the fact that $\| \mathbf{D}_1(n) \|_p - \| \boldsymbol{\mathcal{W}}_1 \|_p \leq \left| \| \mathbf{D}_1(n) \|_p - \| \boldsymbol{\mathcal{W}}_1 \|_p \right| \leq \| \mathbf{D}_1(n) - \boldsymbol{\mathcal{W}}_1 \|_p$,
\begin{align*}
A_{12}(n) &= \frac{n}{k_n} \mathbb{P}\left(\| \mathbf{D}_1(n) \|_p - \| \boldsymbol{\mathcal{W}}_1 \|_p \geq \eta b(n/k_n), \| \boldsymbol{\mathcal{W}}_1 \|_p \leq (y-\eta)b(n/k_n) \right) \\
&\leq \frac{n}{k_n} \mathbb{P}\left(\| \mathbf{D}_1(n) - \boldsymbol{\mathcal{W}}_1  \|_p  \geq \eta b(n/k_n), \| \boldsymbol{\mathcal{W}}_1 \|_p \leq (y-\eta)b(n/k_n) \right) 
\end{align*}
Define  $\boldsymbol{g}_1(n) = \mathbb{E}[\mathbf{D}_1(n) \mid \boldsymbol{\mathcal{W}}_{[n]}]$. Then
\begin{align*}
A_{12}(n) \leq& \frac{n}{k_n}  \mathbb{P}\left(\| \mathbf{D}_1(n) - \boldsymbol{g}_1(n)  \|_p  \geq \eta b(n/k_n)/2, \| \boldsymbol{\mathcal{W}}_1 \|_p \leq (y-\eta)b(n/k_n)\right) \\
&+\frac{n}{k_n} \mathbb{P}\left(\| \boldsymbol{g}_1(n) - \boldsymbol{\mathcal{W}}_1 \|_p \geq \eta b(n/k_n)/2, \| \boldsymbol{\mathcal{W}}_1 \|_p \leq (y-\eta)b(n/k_n) \right) \\
\leq& \frac{n}{k_n}  \mathbb{P}\left(\| \mathbf{D}_1(n) - \boldsymbol{g}_1(n)  \|_p  \geq \eta b(n/k_n)/2, \| \boldsymbol{\mathcal{W}}_1 \|_p \leq (y-\eta)b(n/k_n)\right) \\
&+\frac{n}{k_n} \mathbb{P}\left(\| \boldsymbol{g}_1(n) - \boldsymbol{\mathcal{W}}_1 \|_p \geq \eta b(n/k_n)/2 \right) \\
\equiv& A_{121}(n) + A_{122}(n).
\end{align*}
From Lemmas \ref{lem:Dapprox} and \ref{lem:Dapprox2}, $A_{121}(n) \rightarrow 0$ as $n \rightarrow \infty$. Additionally, Lemma \ref{lem:gapprox} gives that $A_{122}(n) \rightarrow 0$ as $n \rightarrow \infty$. Hence $A_{12}(n) \rightarrow 0$ as $n \rightarrow \infty$. Similarly, for $A_{22}(n)$ 
\begin{align*}
A_{22}(n) =& \frac{n}{k_n} \mathbb{P}\left(\| \boldsymbol{\mathcal{W}}_1 \|_p - \| \mathbf{D}_1(n) \|_p \geq \eta b(n/k_n), \| \boldsymbol{\mathcal{W}}_1 \|_p \geq (y+\eta)b(n/k_n) \right) \\
\leq& \frac{n}{k_n} \mathbb{P}\left(\| \mathbf{D}_1(n) - \boldsymbol{\mathcal{W}}_1  \|_p  \geq \eta b(n/k_n), \| \boldsymbol{\mathcal{W}}_1 \|_p \geq (y+\eta)b(n/k_n) \right) \\
\leq& \frac{n}{k_n} \mathbb{P}\left(\| \mathbf{D}_1(n) - \boldsymbol{g}_1(n)  \|_p  \geq \eta b(n/k_n)/2, \| \boldsymbol{\mathcal{W}}_1 \|_p \geq (y+\eta)b(n/k_n) \right) \\
&+ \frac{n}{k_n} \mathbb{P}\left(\| \boldsymbol{g}_1(n) - \boldsymbol{\mathcal{W}}_1 \|_p \geq \eta b(n/k_n)/2, \| \boldsymbol{\mathcal{W}}_1 \|_p \geq (y+\eta)b(n/k_n) \right) \\
\leq& \frac{n}{k_n} \mathbb{P}\left(\| \mathbf{D}_1(n) - \boldsymbol{g}_1(n)  \|_p  \geq \eta b(n/k_n)/2, \| \boldsymbol{\mathcal{W}}_1 \|_p \geq (y+\eta)b(n/k_n) \right) \\
&+ \frac{n}{k_n} \mathbb{P}\left(\| \boldsymbol{g}_1(n) - \boldsymbol{\mathcal{W}}_1 \|_p \geq \eta b(n/k_n)/2 \right) \\
\equiv& A_{221}(n) + A_{222}(n).
\end{align*}
From Lemmas \ref{lem:Dapprox}, \ref{lem:Dapprox2} and \ref{lem:gapprox}, $A_{221}(n) \rightarrow 0$ and $A_{222}(n) \rightarrow 0$ as $n \rightarrow \infty$. Hence $A_{22}(n) \rightarrow 0$ as $n \rightarrow \infty$. In summary, 
\begin{align*}
\frac{n}{k_n} \mathbb{P}&\left( \| \mathbf{D}_1(n) \|_p > yb(n/k_n),  \| \boldsymbol{\mathcal{W}}_1\|_p \leq yb(n/k_n) \right) \leq A_{11}(n) + A_{12}(n) \rightarrow (y - \eta)^{-\alpha} - y^{- \alpha}, \\
\frac{n}{k_n} \mathbb{P}&\left(\| \boldsymbol{\mathcal{W}}_1\|_p  > yb(n/k_n), \| \mathbf{D}_1(n) \|_p \leq yb(n/k_n) \right) \leq A_{21}(n) + A_{22}(n)\rightarrow y^{-\alpha} - (y+\eta)^{- \alpha},
\end{align*}
as $n \rightarrow \infty$. Since $\eta > 0$ was arbitrary, we have proven \eqref{eq:tailemp1} and \eqref{eq:tailemp2}.
\end{proof}

In Lemma \ref{lem:tailempapproxsub}, we proved an approximation of the tail empirical measure for the MIRG model with $|\mathcal{L}_1| = L$. We use this result, along with the coupling presented in Section \ref{sec:couple} to approximate the tail empirical measure when $|\mathcal{L}_1| < L, |\mathcal{L}_2| \leq L$ and $\alpha > 2$. In order to do so, we first compare $\mathbf{D}_1(n)$ with $\tilde{\mathbf{D}}_1(n)$ presented in Section \ref{sec:couple}.

\begin{lemma}
\label{lem:coupletemp}
Suppose (C1) and(C3) hold. Suppose (C2) holds with $\alpha > 2$. Let $|\mathcal{L}_1| < L, |\mathcal{L}_2| \leq L$. Then 
\begin{align*}
\frac{n}{k_n}  \mathbb{P}\left(\| \mathbf{D}_1(n) - \tilde{\mathbf{D}}_1(n) \|_p  \geq \eta b(n/k_n) \right) \rightarrow 0,
\end{align*}
as $n \rightarrow \infty$.
\end{lemma}

\begin{proof}
It suffices to prove the case where $|\mathcal{L}_2| = L$ since the distance $\| \mathbf{D}_1(n) - \tilde{\mathbf{D}}_1(n) \|_p$ is increasing with the number of Bernoulli-layers. By equivalence of $\ell_p$ norms, we may instead show that as $n \rightarrow \infty$
\begin{align*}
\frac{n}{k_n}  \mathbb{P}\left(\| \mathbf{D}_1(n) - \tilde{\mathbf{D}}_1(n) \|_1  \geq \eta b(n/k_n)/L^{1/p} \right) \rightarrow 0.
\end{align*}
Note that for every $l \in [L]$
\begin{align}
\label{eq:max2}
\frac{n}{k_n}\mathbb{P}\left(W_{1l} > \underset{i \in [n]\setminus \{1\}}{\max} W_{il} \right) = \frac{n}{k_n} \frac{1}{n} = \frac{1}{k_n} \rightarrow 0, \qquad \text{as } n \rightarrow \infty.
\end{align}
Define the event
\begin{align*}
\mathcal{H}_n \equiv \left\lbrace \max_{l \in [L]} W_{1l}\frac{W_{(1)l}(n)}{T_l(n)} < \delta, W_{1l} \leq \underset{i \in [n]\setminus \{1\}}{\max} W_{il} \quad \forall l \in [L] \right\rbrace.
\end{align*}
By Lemma \ref{lem:Ox2} and \eqref{eq:max2}, it suffices to show that as $n \rightarrow \infty$
\begin{align}
\label{eq:couplegoal}
\frac{n}{k_n}  \mathbb{P}\left(\| \mathbf{D}_1(n) - \tilde{\mathbf{D}}_1(n) \|_1  \geq \eta b(n/k_n)/L^{1/p}, \mathcal{H}_n \right) \rightarrow 0,
\end{align}
where $\delta > 0$ is as in \eqref{eq:glin}. Note that by Markov's inequality and \eqref{eq:expcouple},
\begin{align*}
\frac{n}{k_n}  &\mathbb{P}\left(\| \mathbf{D}_1(n) - \tilde{\mathbf{D}}_1(n) \|_1  \geq \eta b(n/k_n)/L^{1/p}, \mathcal{H}_n \right) \\
&= \frac{n}{k_n}  \mathbb{E}\left[ \mathbb{P}\left(\| \mathbf{D}_1(n) - \tilde{\mathbf{D}}_1(n) \|_1  \geq \eta b(n/k_n)/L^{1/p}  \bigg\vert \ \mathbf{W}_{[n]} \right) 1_{\mathcal{H}_n} \right] \\
&\leq \frac{L^{1/p}}{\eta b(n/k_n)}K \frac{n}{k_n} \mathbb{E}\left[ \sum_{l = 1}^L \sum_{j = 1}^n  g_l^2\left(\frac{W_{1l}W_{jl}}{T_{l}(n)}\right)1_{\mathcal{H}_n} \right].
\end{align*}
Note that if $\max_{l \in [L]} W_{1l}\frac{W_{(1)l}(n)}{T_l(n)} < \delta$, by Lemma \ref{lem:sumgsqr}
\begin{align*}
\sum_{l = 1}^L \sum_{j = 1}^n  g_l^2\left(\frac{W_{1l}W_{jl}}{T_{l}(n)}\right) \leq& \sum_{l = 1}^L  C  W^{2}_{1l} \frac{W_{(1)l}(n)}{T_{l}(n)},
\end{align*}
for some constant $C > 0$. Hence
\begin{align*}
\frac{n}{k_n}  &\mathbb{P}\left(\| \mathbf{D}_1(n) - \tilde{\mathbf{D}}_1(n) \|_1  \geq \eta b(n/k_n)/L^{1/p}, \mathcal{H}_n \right) \\
&\leq \frac{L^{1/p}}{\eta b(n/k_n)}KC \frac{n}{k_n} \mathbb{E}\left[ \sum_{l = 1}^L W^{2}_{1l} \frac{W_{(1)l}(n)}{T_{l}(n)} 1_{\left\lbrace \cap_{l \in [L]}\left\lbrace W_{1l} \leq \underset{i \in [n]\setminus \{1\}}{\max} W_{il} \right\rbrace \right\rbrace}  \right],
\end{align*}
and hence it suffices to show that for any $l \in [L]$
\begin{align*}
\frac{1}{b(n/k_n)}\frac{n}{k_n} \mathbb{E}\left[ W^{2}_{1l} \frac{W_{(1)l}(n)}{T_{l}(n)}1_{\left\lbrace \bigcap_{l \in [L]}\left\lbrace W_{1l} \leq \underset{i \in [n]\setminus \{1\}}{\max} W_{il} \right\rbrace \right\rbrace} \right] \rightarrow 0, \qquad \text{as } n \rightarrow \infty.
\end{align*}
See that if $W_{1l} \leq \underset{i \in [n]\setminus \{1\}}{\max} W_{il}$, by similar calculations in the proofs of Lemmas \ref{lem:Wsmall} and \ref{lem:Ox2}
\begin{align*}
 W^{2}_{1l} \frac{W_{(1)l}(n)}{T_{l}(n)} \leq& 2 \frac{W^2_{1l}\underset{i \in [n]\setminus \{1\}}{\max} W_{il}}{\sum_{i \in [n]\setminus \{1\}} W_{il}}.
\end{align*}
and
\begin{align*}
\frac{1}{b(n/k_n)}\frac{n}{k_n}& \mathbb{E}\left[ W^{2}_{1l} \frac{W_{(1)l}(n)}{T_{l}(n)}1_{\left\lbrace \bigcap_{l \in [L]}\left\lbrace W_{1l} \leq \underset{i \in [n]\setminus \{1\}}{\max} W_{il} \right\rbrace \right\rbrace} \right] \\
\leq& \frac{2}{b(n/k_n)}\frac{n}{k_n} \mathbb{E}[W^2_{1l}] \mathbb{E}\left[\frac{\underset{i \in [n]\setminus \{1\}}{\max} W_{il}}{\sum_{i \in [n]\setminus \{1\}} W_{il}}\right] \\
\sim& \frac{2}{b(n/k_n)}\frac{n}{k_n} \frac{\mathbb{E}[W^2_{1l}]}{\mathbb{E}[W_{1l}]} \frac{\mathbb{E}\left[\underset{i \in [n]\setminus \{1\}}{\max} W_{il}\right]}{n - 1} \\
\sim& \frac{2}{b(n/k_n)}\frac{1}{k_n} \frac{\mathbb{E}[W^2_{1l}]}{\mathbb{E}[W_{1l}]} \Gamma\left(1 - \frac{1}{\alpha}\right)b^\star(n),
\end{align*}
where we have applied Theorem 9.1 of \cite{downey2007ratio} and Proposition 2.1 of \cite{resnick2008extreme}. Here, $b^\star(t)$ is the $1-1/t$ quantile function of the distribution of $W_{1l}$ for $t \geq 1$.  Since $b^\star(t)$ is regularly varying with index $1/\alpha$, we may apply (C3) to obtain that bound tends to zero and the proof is complete.
\end{proof}

We now prove the tail empirical measure approximation when $|\mathcal{L}_1| < L, |\mathcal{L}_2| \leq L$ and $\alpha > 2$.

\begin{lemma}
\label{lem:coupletempfinal}
Suppose (C1) and(C3) hold. Suppose (C2) holds with $\alpha > 2$. Let $|\mathcal{L}_1| < L, |\mathcal{L}_2| \leq L$. Then for $y > 0$
\begin{align}
\label{eq:coupletempfinal1}
\frac{n}{k_n}  \mathbb{P}\left(\| \mathbf{D}_1(n) \|_p  > yb(n/k_n), \| \boldsymbol{\mathcal{W}}_1 \|_p  \leq yb(n/k_n) \right) \rightarrow 0, \\
\label{eq:coupletempfinal2}
\frac{n}{k_n}  \mathbb{P}\left(\| \mathbf{D}_1(n) \|_p  \leq yb(n/k_n), \| \boldsymbol{\mathcal{W}}_1 \|_p  > yb(n/k_n) \right) \rightarrow 0,
\end{align}
as $n \rightarrow \infty$.
\end{lemma}

\begin{proof}
We only prove \eqref{eq:coupletempfinal2} as the proof for \eqref{eq:coupletempfinal1} is very similar. See that for $\eta > 0$
\begin{align*}
\frac{n}{k_n}  \mathbb{P}&\left(\| \mathbf{D}_1(n) \|_p  \leq yb(n/k_n), \| \boldsymbol{\mathcal{W}}_1 \|_p  > yb(n/k_n) \right) \\
\leq& \frac{n}{k_n} \mathbb{P}\left(\| \mathbf{D}_1(n) - \boldsymbol{\mathcal{W}}_1 \|_p  > \eta b(n/k_n),\| \boldsymbol{\mathcal{W}}_1 \|_p  > (y + \eta)b(n/k_n) \right)  \\
&+ \frac{n}{k_n}  \mathbb{P}\left(\| \boldsymbol{\mathcal{W}}_1 \|_p  \in (y, y + \eta]b(n/k_n) \right) \\
=& B_1(n) + B_2(n).
\end{align*}
We prove that $B_1(n)$ and $B_2(n)$ converge to $0$ as $n \rightarrow \infty$. For $B_1(n)$, note that
\begin{align*}
B_1(n) \leq& \frac{n}{k_n} \mathbb{P}\left(\| \mathbf{D}_1(n) - \tilde{\mathbf{D}}_1(n) \|_p  > \eta b(n/k_n)/3 \right)\\
&+ \frac{n}{k_n} \mathbb{P}\left(\| \tilde{\mathbf{D}}_1(n) - \tilde{\boldsymbol{g}}_1(n) \|_p  > \eta b(n/k_n)/3 ,\| \boldsymbol{\mathcal{W}}_1 \|_p  > (y + \eta)b(n/k_n) \right) \\
&+ \frac{n}{k_n} \mathbb{P}\left(\| \tilde{\boldsymbol{g}}_1(n) - \boldsymbol{\mathcal{W}}_1 \|_p  > \eta b(n/k_n)/3  \right) \\
=& B_{11}(n) + B_{12}(n) + B_{13}(n),
\end{align*}
where we have defined $\tilde{\boldsymbol{g}}_1(n) = \mathbb{E}[\tilde{\mathbf{D}}_1(n) \mid \mathbf{W}_{[n]}]$. Note that as $n \rightarrow \infty$, $B_{11}(n) \rightarrow 0$ by Lemma \ref{lem:coupletemp}, $B_{12}(n) \rightarrow 0$ by Lemma \ref{lem:Dapprox} and $B_{13}(n) \rightarrow 0$ by Lemma \ref{lem:gapprox}. Hence $B_1(n) \rightarrow 0$ as $n \rightarrow \infty$. From regular variation of $\| \boldsymbol{\mathcal{W}}_1 \|_p$, we additionally have that as $n \rightarrow \infty$
\begin{align*}
B_{2}(n) \rightarrow y^{-\alpha} - (y + \eta)^{-\alpha}.
\end{align*}
Since $\eta > 0$ is arbitrary, the proof is complete
\end{proof}

With Lemma \ref{lem:tailempapprox} in hand, we are now able to prove \ref{thm:weakconvD}. The proof of Theorem \ref{thm:weakconvD} closely follows that of Proposition 3.5 of \cite{bhattacharjee2022large}.

\begin{proof}[Proof of Theorem \ref{thm:weakconvD}]
From Theorem 1.1 of \cite{kallenberg1973characterization}, it suffices to show that for any $d \in \mathbb{N}$
\begin{align}
\label{eq:fdd}
\left(\nu_n(I_1), \dots \nu_n(I_d) \right) \Rightarrow \left(\nu_\alpha(I_1), \dots \nu_\alpha(I_d) \right),
\end{align}
in $\mathbb{R}^d$ as $n \rightarrow \infty$ for intervals of the form $I_k = (a_k, b_k]$, $0 < a_k < b_k \leq \infty$, $k \in [d]$. Let $d$ denote the Euclidean metric. By Slutsky's Theorem, the convergence in \eqref{eq:fdd} is achieved by combining
\begin{align}
\label{eq:fddW}
&\left(\nu^\star_n(I_1), \dots \nu^\star_n(I_d) \right) \Rightarrow \left(\nu_\alpha(I_1), \dots \nu_\alpha(I_d) \right), \qquad \text{in } \mathbb{R}^d, \\
\label{eq:slutsky}
&d\left(\left(\nu_n(I_1), \dots \nu_n(I_d) \right), \left(\nu^\star_n(I_1), \dots \nu^\star_n(I_d) \right)\right) \xrightarrow{p} 0,
\end{align}
$n \rightarrow \infty$, where \eqref{eq:fddW} is given by \eqref{eq:weakconvW}. In order to prove \eqref{eq:slutsky}, it suffices to show that for any $0 < a < b \leq \infty$
\begin{align*}
\nu_n\left( (a, b] \right) - \nu^\star_n\left( (a, b] \right) \xrightarrow{p} 0.
\end{align*}
Towards this end, see that
\begin{align*}
\mathbb{E}\left| \nu_n\left( (a, b] \right) - \nu^\star_n\left( (a, b] \right) \right| =& \mathbb{E}\left| \nu_n\left( (a, \infty] \right) - \nu_n\left( (b, \infty] \right)  - \left( \nu^\star_n\left( (a, \infty] \right) - \nu^\star_n\left( (b, \infty] \right) \right)   \right| \\
\leq& \mathbb{E}\left| \nu_n\left( (a, \infty] \right) -  \nu^\star_n\left( (a, \infty] \right)   \right| + \mathbb{E}\left| \nu_n\left( (b, \infty] \right) -  \nu^\star_n\left( (b, \infty] \right)   \right|.
\end{align*}
Additionally note that for any $y \in (0, \infty]$
\begin{align*}
\mathbb{E}\left| \nu_n\left( (y, \infty] \right) -  \nu^\star_n\left( (y, \infty] \right)   \right| \leq& \frac{1}{k_n} \sum_{i = 1}^n \mathbb{E}\left[1_{\left\lbrace R_i(n) > yb(n/k_n), R^\star_i(n) \leq yb(n/k_n) \right\rbrace} \right] \\
&+ \frac{1}{k_n} \sum_{i = 1}^n \mathbb{E}\left[1_{\left\lbrace R^\star_i(n) > yb(n/k_n), R_i(n) \leq yb(n/k_n) \right\rbrace} \right] \\
=&\frac{n}{k_n}  \mathbb{P}\left(R_1(n) > yb(n/k_n), R^\star_1(n) \leq yb(n/k_n)  \right) \\
&+ \frac{n}{k_n}\mathbb{P}\left( R^\star_1(n) > yb(n/k_n), R_1(n) \leq yb(n/k_n) \right) \\
\rightarrow& 0, \qquad \text{as } n \rightarrow \infty,
\end{align*}
by Lemma \ref{lem:tailempapprox}. Hence we have proved \eqref{eq:slutsky} and the proof is complete.
\end{proof}

\subsection{Consistency of the Hill estimator}\label{sec:hillest}

In this section, we outline the steps required to obtain Theorem \ref{thm:hillc} from Theorem \ref{thm:weakconvD}. The steps are standard and closely follow the proof of Theorem 4.2 of \cite{resnick2007heavy}. 

\begin{proof}[Proof of Theorem \ref{thm:hillc}]
The proof consists of a series of three steps. \\
\textit{Step 1}: We prove that $R_{([k_n])}(n)/b(n/k_n) \xrightarrow{p} 1$ as $n \rightarrow \infty$. The convergence in \eqref{eq:tailemphatD} and inversion \citep[see Proposition 3.2 of][]{resnick2007heavy} gives that as $n \rightarrow \infty$
\begin{align}
\label{eq:ordstaty}
\frac{R_{([k_n y])}(n)}{b(n/k_n)} \xrightarrow{p} y^{-1/\alpha} \qquad \text{in } D(0, \infty].
\end{align}
In particular, as $n \rightarrow \infty$
\begin{align}
\label{eq:ordstat}
\frac{R_{([k_n])}(n)}{b(n/k_n)} \xrightarrow{p} 1.
\end{align}
Additionally, 
\begin{align}
\label{eq:jointnuord}
\left( \frac{1}{k_n} \sum_{i = 1}^n \epsilon_{R_i(n)/b(n/k_n)}, \frac{R_{([k_n])}(n)}{b(n/k_n)} \right) \Rightarrow \left( \nu_\alpha, 1 \right) \qquad \text{in } M_+((0, \infty] \times (0, \infty)),
\end{align}
as $n \rightarrow \infty$. \\
\textit{Step 2}: We now prove the weak convergence of $\hat{\nu}(\cdot)$ in \eqref{eq:weakconvDhat}. We consider an operator $T: M_+((0, \infty]) \times (0, \infty) \mapsto M_+((0, \infty])$ defined by
\begin{align}
T(\mu, c)(A) = \mu(c, A). 
\end{align}
It is shown in the proof of Theorem 4.2 in \cite{resnick2007heavy} that $T$ is continuous at $(\nu_\alpha, 1)$. Hence, applying the continuous mapping theorem to \eqref{eq:jointnuord} gives \eqref{eq:weakconvDhat}. \\
\textit{Step 3}: We now conclude by proving consistency of the Hill estimator. Observe that 
\begin{align*}
H_{k_n, n} = \int_1^\infty \hat{\nu}_n(y, \infty] \frac{dy}{y}. 
\end{align*}
The mapping $f \in D(0, \infty] \mapsto \int_1^M f(y) \frac{dy}{y} \in \mathbb{R}_+$ is almost surely continuous so that \eqref{eq:weakconvDhat} implies that
\begin{align*}
\int_1^M \hat{\nu}_n(y, \infty] \frac{dy}{y} \xrightarrow{p} \int_1^M \nu_\alpha(y, \infty] \frac{dy}{y},
\end{align*}
as $n \rightarrow \infty$. Hence by the second converging together theorem, it suffices to show that for any $\epsilon > 0$
\begin{align*}
\lim_{M \rightarrow \infty} \limsup_{n \rightarrow \infty}\mathbb{P}\left(\int_M^\infty \hat{\nu}_n(y, \infty] \frac{dy}{y} > \epsilon \right) =0,
\end{align*}
or, due to \eqref{eq:ordstat},
\begin{align*}
\lim_{M \rightarrow \infty} \limsup_{n \rightarrow \infty}\mathbb{P}\left(\int_M^\infty \hat{\nu}_n(y, \infty] \frac{dy}{y} > \epsilon, \left| \frac{R_{([k_n y])}(n)}{b(n/k_n)} - 1\right| \leq \eta \right) =0,
\end{align*}
for some $\eta > 0$.  When $\left| \frac{R_{([k_n y])}(n)}{b(n/k_n)} - 1\right| \leq \eta$, however,
\begin{align*}
\hat{\nu}_n(y, \infty] \leq \nu_n(y(1-\eta), \infty].
\end{align*}
Hence by Markov's inequality, it suffices to show that
\begin{align}
\label{eq:expnts}
\lim_{M \rightarrow \infty} \limsup_{n \rightarrow \infty} \mathbb{E}\left[ \int_M^\infty \hat{\nu}_n(y, \infty] \frac{dy}{y} \right] = 0.
\end{align}
See that
\begin{align*}
\mathbb{E}\left[ \int_M^\infty \hat{\nu}_n(y, \infty] \frac{dy}{y} \right] \leq& \mathbb{E}\left[ \int_M^\infty \nu_n(y(1-\eta), \infty] \frac{dy}{y} \right] \\
=& \mathbb{E}\left[ \int_{M(1-\eta)}^\infty \nu_n(y, \infty] \frac{dy}{y} \right] \\
=& \int_{M(1-\eta)}^\infty \frac{n}{k_n}\mathbb{P}\left(R_1 > yb(n/k_n) \right) \frac{dy}{y} \\
\rightarrow& \int_{M(1-\eta)}^\infty y^{-\alpha - 1}dy =  \frac{1}{\alpha((1 - \eta)M)^\alpha},
\end{align*}
which tends to $0$ as $M \rightarrow \infty$. Hence \eqref{eq:expnts} holds and the proof is complete.
\end{proof}


\begin{funding}
D. Cirkovic is supported by NSF Grant DMS-2210735. T. Wang is supported by the National Natural Science Foundation of China Grant 12301660 and the Science and Technology Commission of Shanghai Municipality Grant 23JC1400700.
\end{funding}



\bibliographystyle{imsart-number} 
\bibliography{InhomogeneousMRV.bib}     

\end{document}